\documentclass{amsart}
\usepackage[utf8]{inputenc}
\usepackage[english]{babel}
\usepackage{amsmath}
\usepackage{amsfonts}
\usepackage{amssymb}
\usepackage{dsfont}
\usepackage{amsthm}
\usepackage{hyperref,polynom}
\usepackage{xfrac}
\usepackage{defs}
\usepackage{fonts}
\usepackage{cite}
\usepackage{wrapfig} 
\usepackage[export]{adjustbox} 

\usepackage{thmtools}
\usepackage[backgroundcolor=white, bordercolor=blue, linecolor=blue]{todonotes}
\addto\extrasenglish{}
\addto\extrasenglish{}

\usepackage{mathtools} 
\usepackage{graphicx} 
\usepackage{array} 
\usepackage{enumerate} 
\usepackage{hyperref} 

\usepackage[a4paper,margin=20ex]{geometry}

\newtheorem{theorem}{Theorem}[section]
\newtheorem{proposition}[theorem]{Proposition}
\theoremstyle{definition}
\newtheorem{definition}[theorem]{Definition}
\theoremstyle{plain}
\newtheorem{lemma}[theorem]{Lemma}
\newtheorem{corollary}{Corollary}[theorem]
\theoremstyle{remark}
\newtheorem{remark}[theorem]{Remark}

\makeatletter
\def\namedlabel#1#2{\begingroup
    #2%
    \def\@currentlabel{#2}%
    \phantomsection\label{#1}\endgroup
}
\makeatother

\newcommand{\vphi}{\varphi}

\DeclareMathOperator*{\supp}{supp}

\newcommand{\intdm}[3]{\displaystyle \int\limits_{#1} #2 \, \mathrm{d}#3}
\newcommand{\iintdm}[4]{\iint\limits_{#1}  #2 \, \mathrm{d}#3 \, \mathrm{d}#4}

\newcommand{\diffqu}{\left( ({\bf u}(\bdx)-{\bf u}(\bdy)) \cdot 
\frac{\bdx-\bdy}{|\bdx-\bdy|}\right)}

\newcommand{\diffqbu}{\left( (\bu(\bx)-\bu(\by)) \cdot \frac{\bx-\by}{|\bx-\by|}\right)}
\newcommand{\diffqbv}{\left( (\bv(\bx)-\bv(\by)) \cdot \frac{\bx-\by}{|\bx-\by|}\right)}

\newcommand{\diffqbunorm}{\left| (\bu(\bx)-\bu(\by)) \cdot \frac{\bx-\by}{|\bx-\by|}\right|}

\newcommand{\diffbux}{\left( \bu(\bx) \cdot \frac{\bx-\by}{|\bx-\by|}\right)}
\newcommand{\diffbuy}{\left( \bu(\by) \cdot \frac{\bx-\by}{|\bx-\by|}\right)}
\newcommand{\diffbvx}{\left( \bv(\bx) \cdot \frac{\bx-\by}{|\bx-\by|}\right)}

\newcommand{\shapetensorbxy}{\left( \frac{\bx-\by}{|\bx-\by|} \otimes \frac{\bx-\by}{|\bx-\by|} \right)}

\newcommand{\HRd}{S(\bbR^d;k)}
\newcommand{\HOmegaRd}{S_{\Omega}(\bbR^d;k)}

\usepackage[pagewise]{lineno}\nolinenumbers
\newcommand{\bfs}[1]{\boldsymbol{#1}}
\newcommand{\bdx}{\mathbf{x}}
\newcommand{\bdu}{{\bf u}}

\newcommand{\bdy}{{\bf y}}

\usepackage{eurosym}

\everymath{\displaystyle}

\title[Nonlocal systems related to peridynamics]{Solvability of 
nonlocal systems related to peridynamics}

\author{Moritz Kassmann}
\address{
Universit\"at Bielefeld, Fakult\"at für Mathematik, Postfach 100131, D-33501 
Bielefeld, Germany}
\email{moritz.kassmann@uni-bielefeld.de}
\author{Tadele Mengesha}
\address{
Department of Mathematics,
The University of Tennessee, Knoxville, Tennessee 37996, U.S.}
\thanks{T. Mengesha and J. Scott acknowledge the support of the U.S. NSF under grant DMS-1615726. \\
\indent M. Kassmann acknowledges the support of the German Science Foundation through CRC 1283.
}
\email{mengesha@utk.edu}
\author{James Scott}
\address{
Department of Mathematics,
The University of Tennessee, Knoxville, Tennessee 37996, 
U.S.}
\email{jscott66@vols.utk.edu}

\begin{document}
\begin{abstract}
In this work, we study the Dirichlet problem associated with a strongly coupled 
system of nonlocal equations.  The system of equations comes from a 
linearization of a model of peridynamics, a nonlocal model of elasticity. It is 
a nonlocal analogue of the Navier-Lam\'e system of classical elasticity.  The 
leading operator is an integro-differential operator characterized by a 
distinctive matrix kernel which is used to couple differences of components of a 
vector field. The paper's main contributions are proving well-posedness of the 
system of equations and demonstrating optimal local Sobolev regularity of 
solutions.   We apply Hilbert space techniques for well-posedness. The 
result holds for systems associated with kernels that give rise to 
non-symmetric bilinear forms. The regularity result holds for systems with 
symmetric kernels that may be supported only on a cone. For some specific 
kernels associated energy spaces  are shown to coincide with standard  
fractional Sobolev spaces. 
\end{abstract}

\maketitle

\numberwithin{equation}{section}

\section{Introduction}
\noindent We study the Dirichlet problem associated with a nonlocal system of 
equations  

\begin{equation}\label{intro-main-eqn}
\mathbb{ L} {\bf u} = {\bf f} \quad \text{in } \Omega; \quad 
{\bf u} = {\bfs 0} \quad \text{in } \complement  \Omega.
\end{equation}
where the matrix-valued nonlocal operator $\mathbb{L}$ is of the form  

\begin{align}\label{intro-operator}
\mathbb{L} {\bf u}(\bdx) = \lim\limits_{\veps \to 0^+} 
\intdm{|\bdx-\bdy|>\veps}{k(\bdx,\bdy) \shapetensorbxy 
(\bdu(\bdx)-\bdu(\bdy))}{\bdy},
\end{align} 
\noindent when the limit exists. In the above, $\Omega \subset \bbR^d$ denotes 
an open, bounded set with a sufficiently regular boundary, and $\complement 
\Omega$ denotes its 
complement. The functions ${\bf u}$ and ${\bf f}$  are vector fields defined in 
their respective domain. The kernel $k:\bbR^d \times \bbR^d \to [0,\infty]$ is 
measurable.  For given vectors $\ba = (a_1, a_2,\cdots, a_d), $ and $\bb = (b_1, 
b_2,\cdots, b_d)$, the tensor $\ba \otimes \bb$ is the rank-one matrix with 
$a_ib_j$ as its $ij^{th}$ entry.   From the very definition of the nonlocal 
operator $\mathbb{L}$, it is clear that \eqref{intro-main-eqn} is a strongly 
coupled system of equations.  

The goal of this paper is twofold. First, we formulate a variational 
problem  for \eqref{intro-main-eqn}, the resolution of which provides solutions 
to \eqref{intro-main-eqn}. We treat more general kernels than those covered 
in the literature. For  given data ${\bf f}$ in an appropriate 
class, we describe a notion of solution and demonstrate existence of 
vector-valued solutions $\bdu:\bbR^d \to \bbR^d$ to the nonlocal coupled system 
\eqref{intro-main-eqn}. The second goal is to prove some results related to the 
optimal regularity of solutions. This will be carried out for a specific 
class of 
kernels.  

The motivation to study the above system of equations comes from applications. 
Indeed, the system \eqref{intro-main-eqn} is the equilibrium  
equation  in linearized bond-based peridynamics, a nonlocal continuum model that 
has received a lot of attention in recent years 
\cite{Silling,Silling2010,Silling2007}. 
To describe the model,   a body 
occupying $\Omega\subset \mathbb{R}^{d}$ has undergone the deformation that maps 
a material point $\bdx \in \Omega$ to $\bdx + {\bf u}(\bdx)$ in a deformed 
domain. In this case, the vector field ${\bf u}$ represents the displacement 
field. The peridynamic model treats the body as a complex mass-spring system. 
Any two material points $\bdy$ and $\bdx$ are assumed to be interacting through 
a  bond vector ${\bfs \xi} = \bdy-\bdx$. Under the uniform small strain theory 
\cite{Silling2010}, the strain of the bond $\bdy - \bdx$ is given by the 
nonlocal linearized strain 
\[
\mathcal{D}({\bf u})(\bdx, \bdy) = (\bu(\bx)-\bu(\by)) \cdot \frac{\bx-\by}{|\bx-\by|}. 
\]
A portion of this strain contributes to the volume changing component of the deformation and the
remaining is the shape changing component.  According to the linearized 
bond-based peridynamic model  \cite{Silling2010} the balance of forces is given 
by a system of equations that has the same form as \eqref{intro-main-eqn} for 
some appropriate kernel $k$.  The kernel $k$ contains properties of the modeled 
material and represents the strength and extent of interactions between material 
points $\bdx$ and $\bdy$. The kernel $k$ may depend on $\bdx, \bdy$, their 
relative position $\bdy-\bdx$ or, in the case of homogeneous materials, only on 
their relative distance $|\bdy-\bdx|$. For general $k$, the equation may model 
heterogeneous and anisotropic materials.  The operator $\mathbb{L}{\bf u}$ is 
then the linearized internal force density function due to the deformation 
$\bdx\mapsto\bdx + {\bf u}(\bdx)$ and  is  a weighted average of the linearized 
strain function associated with the displacement ${\bf u}$ \cite{Mengesha-Du, 
Silling2010}. Indeed, rewriting \eqref{intro-operator} in terms of the nonlocal 
strain $\mathcal{D}({\bf u})$ we get 
$\mathbb{L} {\bf u}(\bdx) = \lim\limits_{\veps \to 0^+} 
\intdm{|\bdx-\bdy|>\veps}{k(\bdx,\bdy)\mathcal{D}({\bf u})(\bdx, \bdy) \frac{\bx 
- \by}{|\bx-\by|} }{\bdy}$, whenever it exists. 

The usage of the ``projected" difference of ${\bf u}$, $\mathcal{D}({\bf 
u})(\bdx, \bdy)$,  in   $\mathbb{L}$ makes the operator distinct from other 
nonlocal operators that use the full difference ${\bf u}(\bdy) - {\bf u}(\bdx)$. 
 To see this distinction, it suffices to note that for smooth vector fields  
\[
{\mathcal{D}({\bf u})(\bdx, \bdy) \over |\bdy-\bdx|} =  {(\bdy - 
\bdx)^{\intercal} \big( \varepsilon({\bf u})(\bdx)\big)(\bdy-\bdx) \over |\bdy - 
\bdx|^2}   + o(|\by-\bdx|) 
\]
whereas ${ {\bf u}(\bdy) - {\bf u}(\bdx)\over |\by-\bx|}  = \nabla{\bf u}(\bdx) 
{(\by - \bdx) \over |\by-\bx|}  + o(|\bdy-\bdx|)$, where we have used the 
notation $\varepsilon({\bf u})(\bdx)$ to represent the symmetric part of the 
gradient matrix $\frac{1}{2} ( \nabla{\bf u}(\bdx) +  \nabla{\bf 
u}(\bdx)^{\intercal})$, commonly called the strain tensor. The action $\{ 
\}^\intercal$ denotes the transpose. A consequence of this is that the nonlocal 
system \eqref{intro-main-eqn} can be seen as a nonlocal analogue of the 
Dirichlet problem corresponding to the strongly coupled system of 
partial differential equations  
\[
\text{{\bf div}}\, \mathfrak{C}(\bdx) \varepsilon ({\bf u})(\bdx) = {\bf f} \quad \text{in } \Omega; \quad\quad 
{\bf u} = {\bfs 0} \quad \text{in } \partial \Omega\,, 
\] 
where $\mathfrak{C}(\bdx) $ is a fourth--order tensor of bounded coefficients, 
which is not necessarily uniformly elliptic but rather satisfies the weaker 
Legendre-Hadamard condition. Systems of partial differential equations of the 
above type that are commonly used in the theory of linearized elasticity are 
well studied, see \cite{Giaquinta}.  

Our study of the nonlocal system \eqref{intro-main-eqn} begins with a 
mathematically rigorous understanding of the operator $\mathbb{L}$.  The focus 
is to find a large class of kernels $k$ that may not be symmetric ($k(\bdx, 
\bdy) \neq k(\bdy,\bdx)$), may  have singularity along the diagonal $\{(\bdx, 
\bdy)\in \mathbb{R}^{d}\times \mathbb{R}^{d}:  \bdx = \bdy\}$ or degeneracy on 
some directions such that both the operator $\mathbb{L}$ and associated system 
of equations \eqref{intro-main-eqn} make sense. Notice that even for smooth 
functions the limit in \eqref{intro-operator} does not exist in general 
unless we put a condition on $k$. As with partial differential equations in 
divergence form with measurable coefficients, we study variational 
solutions based on quadratic forms. We use use Hilbert space techniques to study 
the Dirichlet problem 
\eqref{intro-main-eqn}. Applicability of harmonic analysis tools is also 
possible when the system of equations is posed over the entire domain 
$\mathbb{R}^{d}$.

To describe some of our results, following \cite{Kassmann, Schilling} let us 
introduce a decomposition of $k(\bdx, \by)$ in terms of its symmetric 
part $k_{s}$ and its anti-symmetric part $k_{a}$. They are given by   
$$
k_s(\bdx,\bdy) = \frac{1}{2}(k(\bdx,\bdy)+k(\bdy,\bdx)), \quad k_a(\bdx,\bdy) = \frac{1}{2}(k(\bdx,\bdy)-k(\bdy,\bdx)).
$$
Throughout the paper we consider kernels whose symmetric part has locally 
integral second moment, i.e., we assume   
\begin{equation}\label{SM}
\bx \mapsto \intdm{\bbR^d}{\min\{1,  |\bx-\by|^2\}k_s(\bx,\by)}{\by} \in L^1_{loc}(\bbR^d).
\end{equation}
We also define the function space of vector fields  
\[
S(\mathbb{R}^{d};k) = \left\{ {\bf v} \in L^2(\mathbb{R}^{d};\mathbb{R}^{d})):  
\mathcal{D}({\bf v})(\bdx, \bdy)k_s^{1/2}(\bdx,\bdy) \in L^2(\mathbb{R}^{d} 
\times \bbR^d) \right\}.   
\]
The mapping 
$
[ \bu,\bv ]_{H(\mathbb{R}^{d};k)} := \iintdm{\mathbb{R}^{d} \, \bbR^d}{k_s(\bx,\by) \mathcal{D}({\bf u})(\bdx, \bdy) \mathcal{D}({\bf v})(\bdx, \bdy)}{\by}{\bx}
$
defines a bilinear form on $S(\mathbb{R}^{d};k)$.  
One can easily show that the function $\|\cdot\|_{\HRd}$, defined via the relation 
\[
\Vnorm{\bv}^2_{\HRd} = \Vnorm{\bv}^2_{L^2(\bbR^d)} + \iintdm{\bbR^d \, \bbR^d}{k_s(\bx,\by) (\mathcal{D}({\bf v})(\bdx, \bdy))^2 }{\by}{\bx},
\]
serves as a norm for $\HRd$. Moreover, adapting the argument used in the proof of  \cite[Lemma 2.3]{Kassmann}, we can actually show that $S(\mathbb{R}^{d};k)$ is a separable Hilbert space  with inner product $(\cdot, \cdot)_{L^2}  +  [ \cdot,\cdot ]_{S(\mathbb{R}^{d};k)}$ . See also similar results in \cite{Du-Zhou2011, Mengesha-Du, Mengesha-Du-Royal}.  We denote the dual space of $S(\mathbb{R}^{d};k)$ by $S^{*}(\mathbb{R}^{d};k)$.

Roughly speaking, we show the following results: for those kernels $k$ whose 
antisymmetric part is small relative to the symmetric part (e.g. the function 
$\bdx \mapsto \int_\mathbb{R}^{d} \frac{(k_a (\bdx, \bdy))^{2}}{k_{s}(\bdx, 
\bdy)} \, \mathrm{d}\bdy $ is uniformly bounded for any ${\bf u}\in S(\mathbb{R}^{d};k)$),  
the limit in \eqref{intro-operator}  exists  in the weak-* topology of the dual 
space $S^{*}(\mathbb{R}^{d};k)$, and therefore $\mathbb{L}{\bf u} \in 
S^{*}(\mathbb{R}^{d};k)$. This interpretation of the operator allows us to 
define a generalized or weak notion of solution to the system of equations in 
\eqref{intro-main-eqn}. The well-posedness of the problem is demonstrated via 
the application of the Lax-Milgram theorem. To this end, we 
introduce a bilinear form on the space $S(\mathbb{R}^{d};k)\times 
S(\mathbb{R}^{d};k)$ that is compatible with the system \eqref{intro-main-eqn}, 
and by imposing additional conditions on $k$ we show that this form is  
continuous and coercive on appropriate subspaces.  Systems of the type 
\eqref{intro-main-eqn} have been studied extensively in the literature, cf. in 
\cite{Max-Rich, Du-NonlocalCalculus, Du-Navier1, Etienne, Mengesha-Du-Royal}. 
Our results complement the well-posedness result in the above 
cited papers.  Indeed, our work deals with kernels that give rise to 
non-symmetric bilinear forms while earlier works are based  on kernels 
associated to symmetric bilinear forms. As we will see in the next section 
clearly, the non-symmetric bilinear forms we study account for the the presence 
of lower order terms that may involve ``lower order fractional" derivatives, 
while the results in \cite{Mengesha-Du-Royal} deal with linear problems with 
lower order terms that involve the unknown function without any derivatives. 


Let us comment on the case where the vector fields are scalar. In this case, 
the quadratic form under consideration becomes a regular Dirichlet form in the 
sense of \cite{FOT11}. For this reason there is an associated strong Markov 
jump process, which can be used to study the Dirichlet problem. In the 
particular case of translation invariant operators, i.e., when $k(\bdx,\bdy)$ 
depends only on $(\bdx-\bdy)$, the process has stationary independent increments 
and is called L\'{e}vy process. The potential theory of Markov jump processes 
including fine properties of heat kernels has been developed in great detail in 
recent years. It can be shown that our notion of a variational solution 
coincides with the probabilistic notion of harmonicity \cite{Che09, MZZ10} if 
the source term ${\bf f}$ vanishes. For the theory of nonlocal non-symmetric 
Dirichlet forms we refer to \cite{Ma, Fukushima, Schilling}. 
In the case of scalar fields, the variational approach to the 
Dirichlet problem has been used by several authors \cite{SeVa13, Kassmann, 
Rut18}. Note that we only comment on nonlocal operators in bounded 
domains which are related to quadratic forms. For a survey of results on 
nonlocal Dirichlet problem in the non-variational context, see \cite{Ros16}. 

Our study of the nonlocal system \eqref{intro-main-eqn} for general kernels follows the variational 
approach taken in \cite{Kassmann} adapting it to the system of equations. This 
adaptation is not trivial because of the structure of the operator. For 
instance, one can easily check that the seminorm $[\bu, \bu]_{S(\mathbb{R}^{d}; 
k)}$ vanishes over a class of affine maps  of the type $\bu (\bdx) = 
\mathbb{B}\bx + {\bf c}$, where $\mathbb{B}$ is skew-symmetric matrix.   When 
proving coercivity of the form over a subspace, one has to find a mechanism to 
remove this large class of maps, as opposed to constants in the case of 
equations. We will see that we need to use fractional Poincar\'e-Korn-type 
inequalities for the system in contrast to the standard fractional Poincar\'e 
inequality for problems involving scalar fields.    

Let us mention that the system arising in \eqref{intro-main-eqn} is related to 
the Euler-Lagrange system generated by fractional harmonic maps. Those systems 
were studied first in \cite{DaRi11} for the half-Laplacian and then extended to 
more general situations \cite{Sch12, DaL13, Sch15, MiSi15, DaSc17}. In these 
works, the systems arise as Euler-Lagrange equations for critical points of 
functionals like $\| (-\Delta)^{\frac{s}{2}} {\bf u} \|_{L^p}$ for ${\bf u} \in 
{\dot{H}}^{s,p}(\mathbb{R}^d; \mathcal{M})$ where $M \subset \mathbb{R}^N$ is a 
smooth closed manifold. Obviously, these systems are nonlinear in general, 
which makes the regularity theory very challenging. However, the systems 
generated by harmonic maps do not possess a strong coupling in the main part of 
the operator as in \eqref{intro-main-eqn}.

In this paper, we also obtain local regularity results for variational solutions 
of the system \eqref{intro-main-eqn} corresponding to a special class of 
kernels. For this aspect of our study, we concentrate on translation invariant operators with kernels of the form  
 $k(\bx, \by) = k(\bx-\by)$, that is even and comparable with the standard kernel of 
fractional order. We allow this comparability to hold true in any double 
cone $\Lambda$ with apex at the origin, i.e. 
\[
k(\bx-\by) \asymp \frac{1}{|\by-\bdx|^{d + 2s}},\quad s\in (0, 1), \quad \bx-\bdy\in \Lambda. 
\]
For these types of kernels we show that the Hilbert space 
$S(\mathbb{R}^{d}; k)$ is equivalent to the standard fractional Sobolev space 
$H^{s}(\mathbb{R}^{d};\mathbb{R}^{d})$.  Such an equivalence will be proved 
using the Fourier transform. See \cite{Fractional-Hardy, Du-Zhou2011} for 
related results.  For such kernels we show that actually the operator 
$\mathbb{L} : H^{2s}(\mathbb{R}^{d};\mathbb{R}^{d}) \to L^{2}(\mathbb{R}^{d}; 
\mathbb{R}^{d})$ is continuous.  More generally, for any $p\in (1,\infty)$,  if 
we define the non-homogeneous potential space 
\[
\mathcal{L}^{2s, p}(\mathbb{R}^{d}; \mathbb{R}^{d}) = \{{\bf u}\in L^{p}(\mathbb{R}^{d}; \mathbb{R}^{d}) : (-\Delta)^{s}{\bf u}\in L^{p}(\mathbb{R}^{d};\mathbb{R}^{d})\} 
\]
where the fractional Laplacian $(-\Delta)^{s}$ is acting on each component, then it can 
be shown that the nonlocal matrix-valued operator 
$\mathbb{L}$ is continuous from $ 
\mathcal{L}^{2s,p}(\mathbb{R}^{d};\mathbb{R}^{d}) \to L^{p}(\mathbb{R}^{d}; 
\mathbb{R}^{d})$.   Most importantly,  we show in this paper that for any $2 
\leq p \leq \frac{2d}{d-2s}$ and ${\bf f}\in L^{p}(\Omega;\mathbb{R}^{d})$, the 
unique variational solution ${\bf u}\in H^{s}(\mathbb{R}^{d};\mathbb{R}^{d})$ to 
the zero Dirichlet problem 
\begin{equation}\label{intro-main-eqn-ZeroD}
\begin{split}
\mathbb{ L} {\bf u} = {\bf f} \quad \text{in } \Omega, \quad 
{\bf u} = 0 \quad \text{in } \complement  \Omega.
\end{split}
\end{equation}
belongs to $\mathcal{L}^{2s, p}_{loc} (\mathbb{R}^{d}; \mathbb{R}^{d})$.  We say 
${\bf u}\in \mathcal{L}^{2s, p}_{loc}(\mathbb{R}^{d}; \mathbb{R}^{d})$ if ${\bf 
u} \eta \in \mathcal{L}^{2s, p} (\mathbb{R}^{d}; \mathbb{R}^{d})$ for any 
$\eta\in C_{c}^{\infty}(\mathbb{R}^{d})$.  For nonlocal equations,  results of 
the above type have been proved in \cite{Biccari-Warma, Peral, Grubb}.   
We follow an approach that is used in \cite{Biccari-Warma, Biccari-Warma-add}, 
where a similar but more general result is proved for the Dirichlet problem for 
fractional Laplacian equation when the right hand side comes from $L^p$ for any 
$1<p<\infty$.  In the case of vector fields, we could not  cover  all ranges of 
$p$ but only with the additional assumption that the weak solution ${\bf u}\in 
L^{p}$. In the scalar case such an assumption is not necessary since it can 
be proven that a solution to the Dirichlet problem of the fractional Laplacian 
with right hand side in $L^{p}$ must also be in $L^{p}$, see \cite[Lemma 
2.5]{Biccari-Warma}. A similar Calderon-Zygmund type estimate for solutions is 
also proved in \cite[Theorem 16]{Peral}. Unfortunately we are unable to extend 
their  proof to the vector-valued case because the argument in  
\cite{Biccari-Warma} relies on a monotonicity property of an associated 
semigroup and in the case of \cite{Peral} it uses a Moser-type argument where a 
nonlinear function of the solution is used as a test function. Neither of these 
arguments can be applied for systems.

The organization of the paper is as follows: In \autoref{sec:var-sol} we 
introduce additional notation, provide some auxiliary results, and show  
well-posedness of the 
Dirichlet problem \eqref{intro-main-eqn} using Hilbert space methods. We 
present sufficient conditions that imply the validity of fractional Poincar\'e-Korn type estimates for a larger class of kernels. 
We also provide examples of kernels 
for which the theorem is applicable. For a smaller class of kernels we also 
link the energy space $S(\mathbb{R}^{d};k)$ with classical Sobolev spaces. In 
\autoref{sec:regularity} we prove higher-order interior 
regularity of solutions to the 
Dirichlet problem corresponding to a particular class of kernels. 

\section{Variational Formulation}\label{sec:var-sol}

\noindent In this section we set up the variational approach to solve 
the system \eqref{intro-main-eqn} 

\subsection{Notations and Definitions}

Through out the paper we will be using the following functions spaces and their 
associated norm. We assume that $D\subset \mathbb{R}^{d}$ an open subset, and 
$\complement D$ denotes its complement.  
We begin with the function spaces
\[
L^p_{D}(\bbR^d) = \{ \bu \in L^p(\bbR^d;\mathbb{R}^d) \, | \, \bu = {\bfs 0} \text{ a.e. on } \complement D \} \\
\]
which collects $L^{p}$ functions defined over $\mathbb{R}^{d}$ that vanish outside of $D$. 
We also use the notation $S_{D}(\mathbb{R}^{d};k)$ to denote the space of 
functions in $\HRd$ that vanish outside of $D$:  
\[
S_{D}(\mathbb{R}^{d};k) = \{ \bu \in \HRd : \bu = {\bfs 0} \text{ a.e. on } \complement D \}. 
\]
It is not difficult to show that $S_{D}(\mathbb{R}^{d};k)$ is a closed subset of $S(\mathbb{R}^{d};k)$
and that from the definition,
$\left( S_{D}(\mathbb{R}^{d};k), \Vnorm{\cdot}_{\HRd} \right) \hookrightarrow \left( \HRd, \Vnorm{\cdot}_{\HRd} \right)$. 
We denote the dual space of $S_{D}(\mathbb{R}^{d};k)$ by $S^{*}_{D}(\mathbb{R}^{d};k)$. 

To set up a variational problem, we will make necessary preparations. To begin 
with, we introduce a bilinear form that will be used to define a generalized 
notion of a solution to the nonlocal systems of equations.  
\begin{definition}\label{Defn-BF}
Given two measurable functions $\bu$ and $\bv$, we define 
\begin{align*}
\cF^k (\bu,\bv)&= \frac{1}{2} \iintdm{\mathbb{R}^{d}\mathbb{R}^{d}}{k_s(\bx,\by) 
\mathcal{D}({\bf u})(\bdx, \bdy) \mathcal{D}({\bf v})(\bdx, \bdy)}{\by}{\bx} \\ 
&\quad +  \iintdm{\mathbb{R}^{d}\mathbb{R}^{d}}{k_a(\bx,\by) \mathcal{D}({\bf 
u})(\bdx, \bdy) \diffbvx}{\by}{\bx},
\end{align*}
whenever the integrals exist. 
\end{definition}
We notice that the form is not necessarily symmetric. 
We aim to find conditions on $k$ that allow us to have good control on the 
quadratic functional $\cF^k (\bu,\bu)$ for $\bu$ in the function space 
$S(\mathbb{R}^{d};k)$. 
To that end, following \cite{Kassmann,Schilling} let us assume that there exists a symmetric kernel $\tilde{k}$ and constants $A_1 \geq 1$, $A_2 \geq 1$ such that for all $\bx\in \mathbb{R}^{d},$ the measure $|\{\by\in \mathbb{R}^{d}:  k_a^2(\bx,\by) \neq 0 \,\text{and}\, \tilde{k}(\bx,\by) = 0\}|=0$,  and 
\begin{equation}\label{k1}
\iintdm{\bbR^d \, \bbR^d}{\tilde{k}(\bx,\by) (\mathcal{D}({\bf u})(\bdx, \bdy))^2}{\by}{\bx} \leq A_1 \|{\bf u}\|^{2}_{\HRd}
\end{equation}
for all $\bu \in S(\bbR^d;k)$, and that
\begin{equation}\label{k2}
\sup_{\bx \in \bbR^d} \intdm{\bbR^d}{\frac{k_a^2(\bx,\by)}{\tilde{k}(\bx,\by)}}{\by} \leq A_2.
\end{equation}
Note that we can choose $\tilde{k}=k_s$, see \cite{Schilling} where it is used for scalar equations.  
The next lemma describes the proper definition of $\cF^k (\bu,\bv)$ and its continuity on $\HRd$. It also clarifies in what sense the operator \eqref{intro-operator} is defined. 

\begin{proposition}\label{prop:WellDefinedLandF}
Let $\Omega \subset \bbR^d$ be open and assume that $k$ satisfies \eqref{SM} and 
\eqref{k1}-\eqref{k2}. For $n\in\mathbb{N}$, define the subset 
$D_n = \{ (\bx,\by) \in \bbR^d \times \bbR^d \, : \, |\bx-\by| > 1/n \}$ and let
\[
{\mathbb L}_n \bu(\bx) = \intdm{|\bx-\by|>1/n}{k(\bx,\by) \mathcal{D}({\bf 
u})(\bdx, \bdy){\by-\bx \over{|\by-\bx|}}}{\by},
\]
\[
\cF^k_n(\bu,\bv)= \iintdm{D_n}{k(\bx,\by) \mathcal{D}({\bf u})(\bdx, \bdy) 
\diffbvx }{\by}{\bx}. 
\]
Then we have that 
\begin{itemize}
\item[i)] $({\mathbb L}_n \bu,\bv)_{L^2(\bbR^d)} = \cF^k_n(\bu,\bv)$ and 
$\lim_{n \to \infty} \cF^k_n(\bu,\bv)= \cF^{k}(\bu,\bv)$ for all $\bu,\bv \in 
C^{\infty}_c(\Omega)$. 
\item[ii)] Moreover, $\cF^k:S(\bbR^d;k) \times S(\bbR^d;k) \to \bbR$ is 
continuous, and thus on $S_{\Omega}(\bbR^d;k)$. 
\end{itemize}
\end{proposition}

\begin{proof}
We begin by noticing that  if $\bu \in C^{\infty}_c(\bbR^d)$, the expression 
${\mathbb L}_{n}{\bu} (\bx)$ is finite for almost all $\bx\in \mathbb{R}^{d}$. 
This follows from the fact that  for almost all  $(\bx, \by) \in 
\mathbb{R}^{d\times d}$, $k(\bx,\by) \leq k_s(\bx,\by) $, assumption \eqref{SM}, 
and that the integration is over $D_n$. Similarly, for $\bu,\bv \in 
C^{\infty}_c(\bbR^d)$, $\cF^k_n(\bu,\bv)$ is finite as well. 

Now, for $\bu,\bv \in C^{\infty}_c(\bbR^d)$ we have by Fubini's theorem that 
\begin{align*}
({\mathbb L}_n \bu,\bv)_{L^2(\bbR^d)} &= \intdm{\bbR^d}{ \left[ \intdm{\complement 
B(\bx,1/n)}{k(\bx,\by) \shapetensorbxy (\bu(\bx)-\bu(\by))}{\by} \right] \, 
\bv(\bx)}{\bx} \\
&= \iintdm{D_n}{k(\bx,\by) \diffqbu \diffbvx}{\by}{\bx}.
\end{align*}
Split the last integral using the decomposition of $k$ into $k_s$ and $k_a$, and 
interchange $\bx$ and $\by$ to obtain that 
\begin{equation}\label{PV-for-energy}
\begin{split}
({\mathbb L}_n \bu,\bv)_{L^2(\bbR^d)} &= \frac{1}{2} \iintdm{D_n}{k_s(\bx,\by) 
\mathcal{D}({\bf u})(\bdx, \bdy) \mathcal{D}({\bf v})(\bdx, \bdy)}{\by}{\bx} \\
&+ \iintdm{D_n}{k_a(\bx,\by) \mathcal{D}({\bf u})(\bdx, \bdy) 
\diffbvx}{\by}{\bx}.
\end{split}
\end{equation}
We will be using Lebesgue dominated convergence theorem to pass to the limit in 
both term in \eqref{PV-for-energy}. To  pass to the limit in the first term we 
use the function $(\bx,\by) \mapsto {k_s(\bx,\by) \mathcal{D}({\bf u})(\bdx, 
\bdy) \mathcal{D}({\bf v})(\bdx, \bdy)}$ as a majorant. It  is integrable and by 
Cauchy-Schwarz inequality,
\begin{equation}\label{sym-bil}
\begin{split}
 \iintdm{\mathbb{R}^{d}\mathbb{R}^{d}}{k_s(\bx,\by) \mathcal{D}({\bf u})(\bdx, 
\bdy) \, \mathcal{D}({\bf v})(\bdx, \bdy)}{\by}{\bx} \leq 
[\bu,\bu]_{H(\mathbb{R}^{d};k)}[\bv,\bv]_{H(\mathbb{R}^{d};k)} < \infty 
\end{split}
\end{equation}
 due to \eqref{SM}, since $\bu,\bv \in S(\mathbb{R}^{d}, k)$. We next 
bound the integrand in the second term in \eqref{PV-for-energy} as follows. For 
$\bx, \by\in \mathbb{R}^{d}$, using Young's inequality we have that 
\[
\begin{split}
k_a(\bx,\by) \mathcal{D}({\bf u})(\bdx, \bdy) &\diffbvx \leq |k_a(\bx,\by)| 
\tilde{k}^{-1/2}(\bx,\by)|\mathcal{D}({\bf u})(\bdx, \bdy)| \,  |\bv(\bx)| 
\tilde{k}^{1/2}(\bx,\by)\\
&\leq\frac{1}{2}\left(\bv(\bx)^2\frac{k_a^2(\bx,\by)}{\tilde{k}(\bx,\by)} + 
\tilde{k}(\bx,\by) |\mathcal{D}({\bf u})(\bdx, \bdy)|^2 \right), 
\end{split}
\]
where assumption \eqref{k1}-\eqref{k2} guarantees that both functions in the 
right hand side are integrable in the product space 
$\mathbb{R}^{d}\times\mathbb{R}^{d}$.  It is now clear that
\[
\lim_{n\to \infty} ({\mathbb L}_n \bu,\bv)_{L^2(\bbR^d)} = \cF^k(\bu, \bv).
\]
To prove the continuity of the bilinear form $\cF^k : S(\mathbb{R}^{d};k)\times 
S(\mathbb{R}^{d};k) \to \mathbb{R}^{d}$ we estimate the two terms of $\cF^k$ 
separately. As has been shown in \eqref{sym-bil}, the first term of $\cF^k(\bu, 
\bv) $ cannot exceed ${1\over 2} 
[\bu,\bu]_{S(\mathbb{R}^{d};k)}[\bv,\bv]_{S(\mathbb{R}^{d};k)}$. To estimate the 
second term, we use  
 \eqref{k1}-\eqref{k2} with $A = \max\{A_1,A_2\}$ and the Cauchy-Schwarz 
inequality to obtain that 
\begin{align*}
\iint\limits_{\mathbb{R}^{d} \, \mathbb{R}^{d}}  k_a (\bx,\by) \mathcal{D}({\bf 
u})(\bdx, \bdy) &\diffbvx \, \mathrm{d}\by \, \mathrm{d}\bx \\
& \leq \iintdm{\mathbb{R}^{d}\, \mathbb{R}^{d}}{|k_a(\bx,\by)| 
\tilde{k}^{-1/2}(\bx,\by) |\mathcal{D}({\bf u})(\bdx, \bdy)| |\bv(\bx)| 
\tilde{k}^{1/2}(\bx,\by)}{\by}{\bx} \\
& \small{\leq \left( \intdm{\bbR^d}{\bv(\bx)^2 
\intdm{\mathbb{R}^{d}}{\frac{k_a^2(\bx,\by)}{\tilde{k}(\bx,\by)}}{\by}}{\bx} 
\right)^{1/2} \left( \iintdm{\mathbb{R}^{d}\mathbb{R}^{d}}{\tilde{k}(\bx,\by) 
(\mathcal{D}({\bf u})(\bdx, \bdy))^2}{\by}{\bx} \right)^{1/2}} \\
& \leq A \Vnorm{\bv}_{L^2(\bbR^d)} \|{\bf u}\|_{\HRd}.
\end{align*}
Combining the above estimates we have that  
\begin{equation}\label{ContinuityofEEstimates}
|\cF^k(\bu, \bv)|  \leq  {1\over 2}[\bu,\bu]_{S(\mathbb{R}^{d};k)} 
[\bv,\bv]_{S(\mathbb{R}^{d};k)}  + A  \Vnorm{\bu}_{\HRd} 
\Vnorm{\bv}_{L^2(\bbR^d)}  \leq C \Vnorm{\bu}_{\HRd} \Vnorm{\bv}_{\HRd},
\end{equation}proving that $\cF^k$ is indeed a continuous bilinear form on the 
space $S(\mathbb{R}^{d};k)$. 
\end{proof}

\begin{remark}
A discussion on the nature of the ``limiting operator" ${\mathbb L} = \lim_{n\to 
\infty} {\mathbb L}_{n}$ is in order.  First, in the event that the {\em kernel 
$k(\bx, \by)$ is integrable in the sense that if for every $\bx \in \bbR^d$, 
$\intdm{\bbR^d}{k_s(\bx,\by)}{\by} < \infty$ and   the function $\bx \mapsto 
\intdm{\bbR^d}{k_s(\bx,\by)}{\by} \in L^1_{loc} (\bbR^d)$}, then for any $\bu\in 
S(\mathbb{R}^{d};k)$ and for each $n \in \mathbb{N}$, the value ${\mathbb 
L}_{n}\bu(\bx)$ is finite for almost all $\bx\in \mathbb{R}^{d}$ and for almost 
all $\bx\in \mathbb{R}^{d}$ we have
\[
\begin{split}
\lim_{n\to \infty} {\mathbb L}_{n}\bu(\bx) = {\mathbb L}\bu(\bx)&= 
\intdm{\mathbb{R}^{d}}{k(\bx,\by) \mathcal{D}({\bf u})(\bdx, \bdy){\bx-\by 
\over{|\bx-\by|}}}{\by} \\
&=\intdm{\mathbb{R}^{d}}{k_{s}(\bx,\by) \mathcal{D}({\bf u})(\bdx, \bdy){\bx-\by 
\over{|\bx-\by|}}}{\by} +\intdm{\mathbb{R}^{d}}{k_{a}(\bx,\by) \mathcal{D}({\bf 
u})(\bdx, \bdy){\bx-\by \over{|\bx-\by|}}}{\by}. 
\end{split}
\]
Moreover, the above proposition implies that the sequence $\{{\mathbb 
L}_{n}\bu\}$ is bounded in the dual space of $S(\mathbb{R}^{d};k)$, and 
converges in the weak-* topology to $\cF^{k}(\bu, \cdot)$. In this case, since 
for any $\bv\in C_c^{\infty}(\mathbb{R}^{d})$ one can verify using Fubini's 
theorem that 
\[
({\mathbb L}\bu, \bv)_{L^{2}} = \cF^{k}(\bu, \bv) 
\]
we can identify $\cF^{k}(\bu, \cdot)$ with the measurable vector field ${\mathbb 
L}\bu$.  

More generally, for any kernel satisfying \eqref{k1}-\eqref{k2} and \eqref{SM}, 
and for any  $\bu\in S(\mathbb{R}^{d};k)$ one may replace the $L^2$ inner 
product by the duality pairing to define the sequence of functionals $ 
\mathbb{L}_{n} u $ defined by 
\[
\begin{split}
\langle {\mathbb L}_{n} \bu, \bv \rangle &:= \frac{1}{2} 
\iintdm{D_n}{k_s(\bx,\by) \mathcal{D}({\bf u})(\bdx, \bdy) \mathcal{D}({\bf 
v})(\bdx, \bdy)}{\by}{\bx} \\
&\quad + \iintdm{D_n}{k_a(\bx,\by) \mathcal{D}({\bf 
u})(\bdx, \bdy) \diffbvx}{\by}{\bx},
\end{split}
\] 
for $ \bv\in S(\mathbb{R}^{d}, k).$ 
The proposition proved above shows that $\{{\mathbb L}_{n}\bu\}$ is bounded in 
the dual space of $S(\mathbb{R}^{d}, k)$ and converges in the weak-* topology to 
$\cF^{k}(\bu, \cdot)$.  For $\bu\in C_c^{\infty}(\mathbb{R}^{n})$, the limiting 
functional $\cF^{k}(\bu, \cdot)$ can be identified with the function  
\[
\begin{split}
{\mathbb L}\bu(\bx)&= P.V. \intdm{\mathbb{R}^{d}}{ k_{s}(\bx, \by) 
\mathcal{D}({\bf u})(\bdx, \bdy){\bx-\by \over{|\bx-\by|}}}{\by} 
+\intdm{\mathbb{R}^{d}}{ k_{a}(\bx,\by)\mathcal{D}({\bf u})(\bdx, \bdy) {\bx-\by 
\over{|\bx-\by|}}}{\by}.
\end{split}
\]
In the event that $k$ is not integrable and not necessarily symmetric, the 
second term in the above expression corresponds to a term with ``lower order 
derivatives"; see \cite{Kassmann} for a detailed discussion.
\end{remark}

\subsection{The Dirichlet problem of system of nonlocal equations}
In this subsection we use the bilinear form we introduced earlier to define 
what we mean by a variational solution to the Dirichlet problem of the nonlocal 
system of equations.   
\subsubsection{Zero Dirichlet data}
\begin{definition}
Assume that $k$ satisfies both \eqref{SM} and \eqref{k1}-\eqref{k2}. Let $\Omega 
\subset \bbR^d$ be open, bounded. Let ${\bf f} \in S^*_{\Omega}(\bbR^d;k)$.
We say  that $\bu \in \HOmegaRd$ is a solution of
\begin{equation}
\mathbb{ L}\bu = {\bf f}  \quad \text{in } \Omega, \quad\quad
\bu = {\bfs 0}  \quad \text{on } \complement \Omega,
\tag{$D_0$}
\end{equation}
if 
\begin{equation} \label{weak-pairing}
\cF^k(\bu,{\bfs \vphi}) = ({\bf f} ,{\bfs \vphi})_{L^2(\bbR^d)} \quad \text{for 
all } {\bfs \vphi} \in \HOmegaRd.
\end{equation}
\end{definition}
The main result of this section is the following well-posedness of the Dirichlet 
problem $(D_0)$.  
\begin{theorem}\label{thm:SolutionofD0}
Let $\Omega \subset \bbR^d$ be open, bounded. Let $k$ satisfy \eqref{SM} and 
\eqref{k1}-\eqref{k2}. Assume further that 
\begin{itemize}
\item[i)] there exists $C_P \geq 1$ such that for all $\bu \in 
L^2_{\Omega}(\bbR^d)$,
\begin{equation}
\Vnorm{\bu}^2_{L^2(\Omega)} \leq C_P \iintdm{\bbR^d \, \bbR^d}{k_s(\bx,\by) 
(\mathcal{D}({\bf u})(\bdx, \bdy))^2}{\by}{\bx},\, \text{ and } \tag{$PK$}
\end{equation}
\item[ii)] for every $\epsilon>0$, there exists $C_{\epsilon} \geq 0$ such that
\begin{equation}\label{C1}
C_{\epsilon } = \sup_{\bx \in \Omega} \intdm{\complement 
B(\bx,\veps)}{|k_{a}(\bdx, \bdy)|}{\by}  < \infty, \, \text{ and }
\end{equation}
\item [iii)]
\begin{equation}\label{C2}
\inf_{\bx \in \bbR^d} \liminf_{\veps \to 0^+} \intdm{\complement 
B(\bx,\veps)}{k_a(\bx,\by) \shapetensorbxy }{\by} \geq 0
\end{equation}
in the sense of quadratic forms.
\end{itemize}
Then corresponding to any ${\bf f} \in S^{\ast}_{\Omega}(\bbR^d;k)$ 
 there exists a unique solution $\bu \in \HOmegaRd$ to $(D_0)$. Moreover, there 
exists a constant $c>0$ such that 
 \[
 [{\bf u}, {\bf u}]_{\HOmegaRd} \leq c \|{\bf 
f}\|_{S^{*}_{\Omega}(\mathbb{R}^{d}; k)}.
 \]
\end{theorem}

\begin{remark}
Condition $(PK)$ in the theorem is called a Poincar\'e-Korn type inequality. In 
the theorem it appears as an assumption that restricts the choice of the kernel 
$k$. Later, we provide sufficient conditions that guarantee the validity of 
$(PK)$ for  a class of kernels. Conditions \eqref{C1}-\eqref{C2} should be 
treated as cancellation conditions on the antisymmetric part of the kernel. 
Indeed, condition \eqref{C1} is an integrability requirement on $k_a$ away from 
the diagonal which allows us to apply Fubini's theorem and use other properties 
of the integral.    
Condition \eqref{C2} on the other hand says that the term in the energy 
$\cF^{k}[{\bf u},{\bf u}]$ involving the antisymmetric part $k_{a}(\bdx, \bdy)$ 
should  not be ``too negative." This condition can be relaxed a little bit, but 
verifying it may be a challenge. A relaxed condition is given in \cite[Remark 
3.3]{Kassmann}. See also nonlocal variational problems that involve sign 
changing kernels in a different sense in \cite{Mengesha-Du-Royal}.
\end{remark}

\begin{proof}[Proof of \autoref{thm:SolutionofD0}] 
We use the Lax-Milgram theorem to prove the result.  Conditions 
\eqref{C1}-\eqref{C2} will be used to show that $\cF^{k}[{\bf u},{\bf u}]$ is 
positive semidefinite, while $(PK)$ implies positive definiteness of the 
energy. 
We show step by step that all the assumptions in the Lax-Milgram theorem are 
satisfied. We begin by noting that as in \autoref{prop:WellDefinedLandF} the 
conditions \eqref{SM}, \eqref{k1}-\eqref{k2} imply that the bilinear form 
$\mathcal{F}^{k}$ is a continuous form on $S(\mathbb{R}^{d};k)$.  Next, we will 
show that $\mathcal{F}^{k}$ is coercive on the closed subspace 
$S_{\Omega}(\mathbb{R}^{d}; k)$ of $S(\mathbb{R}^{d}; k)$. 
We begin by showing that   
\begin{equation}\label{eq:PositiveDefinitenessofE}
\cF^k(\bu,\bu) \geq {1\over 2}[\bu,\bu]_{\HRd} \quad \text{for all } \bu \in 
\HOmegaRd.
\end{equation}
For any  $\veps > 0$, and for any $\bu \in \HOmegaRd$ we have that 
\begin{align*}
&\iintdm{|\bx-\by|>\veps}{\mathcal{D}({\bf u})(\bdx, \bdy) \diffbuy 
k_a(\bx,\by)}{\by}{\bx}\\&= 
\frac{1}{2}\iintdm{|\bx-\by|>\veps}{\mathcal{D}({\bf u})(\bdx, \bdy) \left( 
(\bu(\bx)+\bu(\by)) \cdot \frac{\bx-\by}{|\bx-\by|} \right) 
k_a(\bx,\by)}{\by}{\bx} \\
&= \frac{1}{2}\iintdm{|\bx-\by|>\veps}{\left( \diffbux^2 - \diffbuy^2 
\right)k_a(\bx,\by)}{\by}{\bx},  
\end{align*}
where we have used the anti-symmetry of $k_{a}$ in the first equality. 
We use the integrability assumption \eqref{C1} in the last integral to  apply 
Fubini's theorem to obtain that 
\begin{align*}
\iintdm{|\bx-\by|>\veps}{\mathcal{D}({\bf u})(\bdx, \bdy) &\diffbuy 
k_a(\bx,\by)}{\by}{\bx}= \iintdm{|\bx-y|>\veps}{\diffbux^2 
k_a(\bx,\by)}{\by}{\bx}\\
&= \intdm{\bbR^d}{\bu(\bx)^{\intercal} \left[ \intdm{\complement 
B(\bx,\veps)}{k_a(\bx,\by) \shapetensorbxy}{\by} \right] \, \bu(\bx) }{\bx}.
\end{align*}
We then conclude from \eqref{C2} that   
\begin{align*}
\iintdm{\bbR^d \, \bbR^d}{\mathcal{D}({\bf u})(\bdx, \bdy) & \diffbuy 
k_a(\bx,\by)}{\by}{\bx} \\
&= \lim_{\veps \to 0} 
\iintdm{|\bx-\by|>\veps}{\mathcal{D}({\bf u})(\bdx, \bdy) \diffbuy 
k_a(\bx,\by)}{\by}{\bx} \\
&= \intdm{\bbR^d}{\bu(\bx)^{\intercal} \left[ \intdm{\complement 
B(\bx,\veps)}{k_a(\bx,\by) \shapetensorbxy}{\by} \right] \, \bu(\bx)}{\bx} \geq 0.
\end{align*}
Hence, from the definition of the bilinear form we have that 
\begin{equation}
\cF^k(\bu,\bu) \geq \frac{1}{2}\iintdm{\bbR^d \, \bbR^d}{(\mathcal{D}({\bf 
u})(\bdx, \bdy))^2 k_s(\bx,\by)}{\by}{\bx}
\end{equation}
and \eqref{eq:PositiveDefinitenessofE} is proved.
Therefore, by the Poincar\'e-Korn inequality $(PK)$ 
and \eqref{eq:PositiveDefinitenessofE},
\begin{align*}
\cF^k(\bu,\bu) \geq \frac{1}{4C_P}\Vert \bu \Vert^2_{L^2(\Omega)} + 
\frac{1}{4}[\bu,\bu]_{\HRd} \geq \frac{1}{4 C_P}\Vert \bu \Vert^2_{\HRd}.
\end{align*}
Finally, the Lax-Milgram Lemma implies that there exists a unique $\bu \in 
\HOmegaRd$ such that 
$$
\cF^k (\bu,{\bfs \vphi}) = \Vint{{\bf f}, \vphi} \qquad \text{for all } {\bfs 
\vphi} \in \HOmegaRd.
$$
\end{proof}
\subsubsection*{A Sufficient Condition for the Poincar\'e-Korn Inequality}
We emphasize that the Poincar\'e-Korn inequality $(PK)$ is an assumption in 
\autoref{thm:SolutionofD0}. Here we present a theorem that gives sufficient 
conditions on the kernel $k$ for the validity of thee Poincar\'e-Korn 
inequality. 
Given $\mathcal{I}$ an open subset of the unit sphere $\mathbb{S}^{d-1}$ such 
that the Hausdorff measure $\mathcal{H}^{d-1}(\mathcal{I}) > 0$, we call the set 
$\Lambda = \left\lbrace {\bf h}\in \mathbb{R}^{d}\setminus\{{\bfs 0}\} \, : \, 
\frac{{\bfs h}}{{|{\bf h}|} } \in \mathcal{I} \cup (- \mathcal{I}) 
\right\rbrace$ a double cone  with apex at the origin. Note that for any such 
cone $\Lambda = -\Lambda$. Denote $\Lambda_{B_r}:= \Lambda \cap B_{r}({\bfs 
0})$, a part of a double cone with apex at the origin in $B_{r}({\bfs 0})$.

\begin{proposition}\label{lem:SufficientPoincareLemma}
Let $\Omega \subset \bbR^d$ be open, bounded. Assume that there is an even, 
nonnegative function $\rho \in L^1(\bbR^d)$ satisfying the following 
conditions: 
\begin{enumerate}
\item[i)] There exists $\delta_{0} > 0$ and a cone $\Lambda$ with apex at the 
origin such that $\Lambda_{B_{\delta_{0} }} \subset \{\rho>0\}$.
\item[ii)] There exists $c_0 > 0$ such that for all ${\bf u} \in \HOmegaRd$
\begin{align}
\begin{split}
\iintdm{\bbR^d  \, \bbR^d}{& k_s(\bx,\by) \diffqbu^2}{\by}{\bx} \\
& \geq c_0 \iintdm{\bbR^d \, \bbR^d}{\rho(\bx-\by) \diffqbu^2}{\by}{\bx} \,.
\end{split}
\end{align}
\end{enumerate}
Then, there exists $C_P = C_P(\Omega,c_0,\rho,\delta_{0}, \Lambda) > 0$ such 
that for all ${\bf u} \in \HOmegaRd$
\begin{equation}
\Vnorm{\bu}_{L^2(\bbR^d)} \leq C_P \iintdm{\bbR^d \, \bbR^d}{k_s(\bx,\by) 
\diffqbu^2}{\by}{\bx}.
\end{equation}
\end{proposition}

In the next subsection we give a number of examples of kernels that satisfy the 
hypothesis of the proposition. We need the following lemma which generalizes 
\cite[Proposition 1.2]{Temam} and \cite[Lemma 2.2]{Du-Zhou2011} that give a 
characterization of infinitesimal rigid motions. 
\begin{lemma}\label{lem:char-RM}
Suppose that ${\bf u}:\mathbb{R}^{d}\to \mathbb{R}^{d}$ is a measurable 
vector-valued function such that for some fixed $\delta_0 >0$, and 
$\mathcal{J}\subset \mathcal{S}^{d-1}$, an open subset of the unit sphere 
$\mathbb{S}^{d-1}$ with  $\mathcal{H}^{d-1}(\mathcal{J}) > 0$, it holds that 
for almost every $\bx \in \bbR^d$,
$$
(\bu(\bx)-\bu(\by)) \cdot (\bx-\by) = 0 \quad \text{for almost every ${\bf y} 
\in \left\lbrace {\by }\in B_{\delta_0}(\bx): {{\bx-\by}\over|\bx-\by|} \in 
\mathcal{J} \right\rbrace $.}
$$
Then ${\bf u}$ is an affine map of the form ${\bf u}(\bx) = \mathbb{A}\bx + 
{\bf b}$, almost everywhere,  where $\mathbb{A}$ is a constant skew symmetric 
matrix ($\mathbb{A}^{\intercal} = -\mathbb{A})$, and ${\bf b}\in 
\mathbb{R}^{d}$. 
\end{lemma}
\begin{proof}
For $\bx \in \bbR^d$, define $\Gamma(\bx) =  \left\lbrace {\by }\in 
B_{\delta_0}(\bx): {{\bx-\by}\over|\bx-\by|} \in \mathcal{J} \right\rbrace$.  
For each $\bx$, the set $\Gamma(\bx)$ is an open set and is in fact the 
intersection of the ball $B_{\delta_0}(\bx)$ with the cone whose directions lie 
in $\mathcal{J}$ with apex at $\bx$.
Let $\{{\bf e}_i\}_{i=1}^d$ denote a basis for $\bbR^d$ contained in  
$\mathcal{J}$; such a basis exists because $\mathcal{J}$ is nontrivial.  
Then since the Lebesgue integral is continuous with respect to translations, 
there exists a $\delta_1 > 0$ such that the function
$$
\delta \mapsto \intdm{\bbR^d}{\left( \prod_{i=1}^d \chi_{\Gamma(\delta {\bf 
e}_i)} \right) \chi_{\Gamma(0)}}{\bx}
$$
is positive. For $\bx \in \bbR^d$, set $\tilde{\Gamma}(\bx) := \left( 
\bigcap_{i=1}^d \Gamma(x+\delta_1 {\bf e}_i) \right) \cap \Gamma(\bx)$. By the 
discussion above, $\tilde{\Gamma}(\bx)$ is an open set of positive measure. 

Now fix $\bx_0 \in \bbR^d$ (up to a set of measure zero). Then by the main 
assumption in the lemma, for almost every $\bx \in \tilde{\Gamma}(\bx_0)$ we 
have
\begin{equation}\label{eq-PoincareTechnicalLemma2-ProofEq1}
\left( (\bu(\bx)-\bu(\bx_0)) \cdot (\bx-\bx_0) \right) = 0
\end{equation}
and
\begin{equation}\label{eq-PoincareTechnicalLemma2-ProofEq2}
\left( (\bu(\bx)-\bu(\bx_0 + \delta_1 {\bf e}_i)) \cdot (\bx-\bx_0 - \delta_1 
{\bf e}_i) \right) = 0\,.
\end{equation}
Therefore, adding and subtracting $\bu(\bx_0)$ in the first argument of 
\eqref{eq-PoincareTechnicalLemma2-ProofEq2} and $\bx_0$ in the second and using 
\eqref{eq-PoincareTechnicalLemma2-ProofEq1} we see that
$$
(\bu(\bx)-\bu(\bx_0)) \cdot \delta_1 {\bf e}_i = - (\bu(\bx_0+\delta_1 {\bf 
e}_i)-\bu(\bx_0)) \cdot (\bx-\bx_0).
$$
So,
$$
\bu(\bx) \cdot {\bf e}_i = \frac{1}{\delta_1} \left( ({\bf u}(\bx_0+\delta_1 
{\bf e}_i)-{\bf u}(\bx_0)) \cdot (\bx-\bx_0) \right) + \bu(\bx_0) \cdot {\bf 
e}_i
$$
for every $\bx \in \tilde{\Gamma}(\bx_0)$ and every $i$, which is clearly a 
linear map.
Then, letting $\mathbb{E} = [{\bf e}_i]$ be the matrix of basis vectors, and 
${\bf u} = (u_1, u_2,\cdots,u_d)$, we have that 
\begin{align*}
u_i(\bx) = \left( \mathbb{E}^{-1} (\mathbb{E}\bu) \right)_i = \sum_j 
e^{-1}_{ij} ({\bf e}_{j}\cdot \bu(\bdx))
\end{align*}
which, being a sum of of linear maps, is still linear. We conclude that for 
almost all $\bx\in \tilde{\Gamma}(\bx_0)$ the vector field $\bu$ is of the form 
$\mathbb{A}(\bx_0)\bx + {\bf b}(\bx_0)$, where $\mathbb{A}$ is matrix with 
constant entries (depending possibly on $\bx_0$) and ${\bf b}$ is a constant 
vector (also depending on $\bx_0$) in $\bbR^d$. 

Next given any two points in $\bbR^d$, outside of a set of measure zero, we  
connect them by  finitely many sets of the form $\tilde{\Gamma}(\bx)$, i.e. for 
any two points $\bx_0$ and $\bx_1$ in $\bbR^d$ there exists a finite subcover 
of $(\tilde{\Gamma}(\bx))_{\bx \in \bbR^d}$, denoted 
$(\tilde{\Gamma}(\bx_k))_{k=1}^N$, such that $\tilde{\Gamma}(\bx_k) \cap 
\tilde{\Gamma}(\bx_{k+1}) \neq \emptyset$ and $\bx_0 \in 
\tilde{\Gamma}(\bx_0)$, $x_1 \in \tilde{\Gamma}(\bx_N)$. This is possible, 
since the line segment connecting $\bx_0$ and $\bx_1$ is compact. Therefore the 
$\bu$ given above is the same in neighboring intersecting open sets  and so $\bu 
= \mathbb{A}\bx + {\bf b}$ on $\bbR^d$ where $\mathbb{A}$, ${\bf b}$ are now 
constants.  Again from the main assumption, the matrix $\mathbb{A}$ must be 
skew symmetric.  
\end{proof}
\begin{corollary}\label{cor:PoincareTechnicalLemma2}
Let $\Omega \subset \bbR^d$ be open, bounded. Assume that there is a 
nonnegative even function $\rho \in L^1(\bbR^d)$ satisfying the following:
\begin{center}
There exist $\delta_0 > 0$ and a symmetric cone $\Lambda$ with vertex at the 
origin such that $\Lambda \cap B_{\delta_0}(0) \subset \supp \rho$.
\end{center}
Suppose that $\bu \in L^2(\bbR^d;\mathbb{R}^{d})$ satisfies
$$
\iintdm{\bbR^d \, \bbR^d}{\rho(\bx-\by) \diffqbu^2}{\by}{\bx} = 0.
$$
Then $\bu = {\bf 0}$ almost everywhere.
\end{corollary}
\begin{proof}
Since the integrand is nonnegative, we see that for almost every $\bx \in 
\bbR^d$,
$$
(\bu(\bx)-\bu(\by)) \cdot (\bx-\by) = 0
$$
for almost every $\by \in \supp \rho + \bx := \{ \bz \, : \, \bz-\bx \in \supp 
\rho \}$. By assumption and \autoref{lem:char-RM}, ${\bf u}$ is an affine map. 
But since ${\bf u} \in L^2(\bbR^d)$, it follows that $\bu$ must be the zero 
vector field. 
\end{proof}

Now, we are ready to prove the sufficiency for Poincar\'e-Korn inequality.  The 
proof follows the argument presented in the proof of  \cite[Proposition 
2]{Mengesha-Du}, that applies the case when $\rho$ is radial. 

\begin{proof}[Proof of \autoref{lem:SufficientPoincareLemma}]
Without loss of generality we assume that $\rho$ has compact support of 
positive measure. (else replace $\rho$ by $\rho \chi_{B(0,r)}$). Then $\rho$ 
satisfies \eqref{SM}.  
To prove the lemma, it suffices to show that there exists a constant $C>0$ such 
that for all ${\bf u}\in L^{2}_{\Omega}(\mathbb{R}^{d})$, 
\[
\|{\bf u}\|_{L^{2}} \leq C [{\bf u}, {\bf u}]_{H(\mathbb{R}^{d};\rho)}. 
\]
Suppose to the contrary; that there exists $\{ \bu_n \} \subset 
S_{\Omega}(\bbR^d;\rho)$ such that $\forall \, n \in \bbN$ 
$\Vnorm{\bu_n}_{L^2(\bbR^d)} = 1$ and $[\bu_n,\bu_n]_{S(\mathbb{R}^{d};\rho)} 
\to 0$ as $n \to \infty$. 
Let $\bu$ be the weak $L^2(\bbR^d)$ limit of $\{ \bu_n \}$. We first show that 
$\bu = {\bfs 0}$. Note that because of the properties of $\rho$ the operator
$$
\mathbb{L}_{\rho} \bu(\bx) := \intdm{\bbR^d}{\rho(\bx-\by) \shapetensorbxy 
(\bu(\bx)-\bu(\by))}{\by}
$$
is a bounded linear map from $L^2(\mathbb{R}^{d};\mathbb{R}^{d})$ to 
$L^2(\mathbb{R}^{d};\mathbb{R}^{d})$.
Let ${\bfs \vphi} \in C^{\infty}_c(\bbR^d)$. Then, by the Cauchy-Schwartz 
inequality,
\begin{align*}
\cF^{\rho}(\bu_n,{\bfs \vphi}) &= \iintdm{\bbR^d \, \bbR^d}{\rho(\bx-\by) 
\left( (\bu_n(\bx)-\bu_n(\by)) \cdot {\bx - \by \over |\by-\bx|} \right) \left( 
{\bfs \vphi}(\bx) \cdot {\bx - \by \over |\by-\bx|} \right)}{\by}{\bx} \\
&\leq \left( \iintdm{\bbR^d \, \bbR^d}{\rho(\bx-\by) \left( 
(\bu_n(\bx)-\bu_n(\by))\cdot {\bx - \by \over |\by-\bx|} \right)^2}{\by}{\bx} 
\right)^{1/2} \left( \iintdm{\bbR^d \, \bbR^d}{\rho(\bx-\by)|{\bfs 
\vphi}(\bx)|}{\by}{\bx}\right)^{1/2} \\
&=\Vnorm{\rho}_{L^1(\bbR^d)} [\bu_n,\bu_n]^{1/2}_{S(\bbR^d;\rho)} \Vnorm{{\bfs 
\vphi}}_{L^2(\bbR^d)} \\
&\to 0 \text{ as } n \to \infty.
\end{align*}
Now, since $\rho$ is symmetric, $\cF^{\rho}$ is symmetric. Thus,
$$
(\mathbb{L}_{\rho} \bu_n,{\bfs \vphi})_{L^2} = \cF^{\rho}(\bu_n,{\bfs \vphi}) = 
\cF^{\rho}({\bfs \vphi},\bu_n) = (\bu_n, \mathbb{L}_{\rho} {\bfs \vphi})_{L^2} 
\quad \forall \, n \in \bbN.
$$
Since $\mathbb{L}_{\rho} {\bfs \vphi} \in L^2(\mathbb{R}^{d};\mathbb{R}^{d})$, 
it follows that $\cF^{\rho}({\bfs \vphi},\bu_n) \to \cF^{\rho}({\bfs 
\vphi},\bu)$ as $n \to \infty$.
Therefore, for all ${\bfs \vphi} \in C^{\infty}_c(\bbR^d;\mathbb{R}^{d})$, 
$(\mathbb{L}_{\rho} \vphi, \bu)_{L^2} = (\mathbb{L}_{\rho} \bu, {\bfs 
\vphi})_{L^2} = 0$. Thus $\mathbb{L}_{\rho} \bu = {\bfs 0}$ a.e. 
As a consequence, since $\rho$ is even and assumption $ii)$
$$
(\mathbb{L} \bu,\bu)_{L^2} = \cF^{\rho}(\bu,\bu) = 
\frac{1}{2}[\bu,\bu]_{S(\bbR^d;\rho)} = \frac{1}{2} \iintdm{\bbR^d \, 
\bbR^d}{\rho(\bx-\by) \diffqbu^2}{\by}{\bx}= 0.
$$
We can now apply \autoref{cor:PoincareTechnicalLemma2} to conclude that 
$\bu \equiv {\bfs 0}$ on $\bbR^d$.

Next we show that in fact, up to a subsequence, $\bu_n \to 0$ strongly in 
$L^2$, and arrive at our contradiction. To show this it suffices to demonstrate 
that $\|{\bf u}_{n}\|_{L^{2}(\Omega)} \to 0$ as $n\to \infty$.  Define 
$\bbK({\bfs \xi}) = \rho({\bfs \xi}) {{\bfs \xi}\otimes {\bfs \xi}\over |{\bfs 
\xi}|^{2}} $. Note that $\bbK \in L^1(\bbR^d; \bbR^d\times \bbR^d)$ since $\rho 
\in L^1 \cap L^{\infty}(\bbR^d)$. Then, define $\bbK \ast \bu(\bx)$ and 
$\mathbb{B}$ as
$$
(\bbK \ast \bu(\bx))_i = \sum_{j=1}^d 
\intdm{\bbR^d}{(\bbK(\bx-\by))_{ij}\bu(\by)_j}{\by}, \qquad {\mathbb{B} =  
\intdm{\bbR^d}{\bbK({\bfs \xi})}{{\bfs \xi}}}.
$$
Both quantities converge absolutely and are well-defined. Further, 
$\mathbb{L}_{\rho} \bu(\bx) = \bbK \ast \bu(\bx) - \bbB \bu(\bx)$. Note that 
$\bbB$ is a positive definite constant matrix, which follows from the fact that 
 $
\Phi(v) = {\bv}^{\intercal} \mathbb{B} \bv = \intdm{\bbR^d}{\rho({\bfs \xi}) 
\left| \frac{{\bfs \xi}}{|{\bfs \xi}|} \cdot {\bv} \right|^2}{{\bfs \xi}}
$
is a continuous and positive function on the units sphere $\mathbb{S}^{d-1}$.   
From an above estimate, we have that
$$
(\bbL_{\rho} \bu_n , \bu_n)_{L^2} \leq \Vnorm{\rho}_{L^1(\bbR^d)} 
[\bu_n,\bu_n]_{S(\bbR^d;\rho)} \Vnorm{\bu_n}_{L^2(\bbR^d)} \to 0 \quad \text{ 
as } n \to \infty.
$$
Since $\bu_n \rightharpoonup {\bfs 0}$ weakly in $L^2(\bbR^d)$, by compactness 
of the convolution operator \cite[Corollary 4.28]{Brezis-Book} we have that
$$
\bbK \ast \bu_n(\bx) \to {\bfs 0} \text{ strongly in } L^2(\Omega; \bbR^d).
$$
Therefore, if $\bbB \geq \gamma \bbI$ in the sense of quadratic forms, we have 
that
\begin{align*}
\gamma \lim\limits_{n \to \infty} \intdm{\bbR^d}{|\bu_n|^2}{\bx} &\leq 
\lim\limits_{n \to \infty} (\bbB \bu_n,\bu_n)_{L^2(\mathbb{R}^{d})}
=\lim\limits_{n \to \infty} (\bbB \bu_n,\bu_n)_{L^2(\mathbb{R}^{d})} + 
\lim\limits_{n \to \infty} (\bbK \ast \bu_n, \bu_n)_{L^{2}(\Omega)} \\
&= \lim\limits_{n \to \infty} (\bbL_{\rho} \bu_n,\bu_n)_{L^{2}(\Omega)} = 0.
\end{align*}
That completes the proof. 
\end{proof}

\subsubsection{Examples of kernels}
There are several examples of kernels that satisfy all the conditions of the 
theorem; a number of them are discussed in detail in \cite{Kassmann} in 
connection with the solvability of the Dirichlet problem of nonlocal equations. 
For some of these examples, the verification of $(PK)$ is nontrivial.  We list 
several examples of nontrivial kernels, for which we can verify all the 
conditions. This shows that the nonlocal Dirichlet problem for the 
corresponding of equations is well-posed. Given $\mathcal{I}$ an open subset of 
the unit sphere $\mathbb{S}^{d-1}$ such that Hausdorff measure 
$\mathcal{H}^{d-1}(\Lambda) > 0$, we call the set $\Lambda = \left\lbrace {\bf 
h}\in \mathbb{R}^{d}\setminus\{{\bfs 0}\}: \frac{{\bfs h}}{{|{\bf h}|} } \in 
\mathcal{I} \cup (- \mathcal{I}) \right\rbrace$ a double cone  with apex at the 
origin. Note that for such cone $\Lambda = -\Lambda$. Denote $\Lambda_{B_r}:= 
\Lambda \cap B_{r}({\bfs 0})$, a part of a double cone with apex at the origin 
in $B_{r}({\bfs 0})$. 

\textit{Example 1:} Suppose that $\rho({\bfs \xi})$ is a nonnegative, even,   integrable 
function in $\mathbb{R}^{d}$.  Define now
$$
k(\bdx, \bdy) = \rho(\bdx-\bdy)\chi_{\Lambda_{B_1}}(\bx-\by)\,.
$$
Since $K$ is symmetric, we only need to verify the Poincar\'e-Korn type 
inequality $(PK)$. But this is a consequence of 
\autoref{lem:SufficientPoincareLemma}. See also \cite[Proposition 
2]{Mengesha-Du} for a similar result that is valid for radial kernels.   Note 
that for these types of kernels the space $S(\mathbb{R}^{d};k)$ is just 
$L^{2}(\mathbb{R}^{d}; \mathbb{R}^{d})$.

\textit{Example 2:} More generally, if $\mathcal{C} = \left\lbrace {\bf h}\in B_{1}({\bfs 
0}): \frac{{\bfs h}}{{|{\bf h}|} } \in \mathcal{J} \right\rbrace$, and 
$\mathcal{J}$
 is an open subset of the unit sphere $\mathbb{S}^{d-1}$ with Hausdorff measure 
$\mathcal{H}^{d-1}(\Lambda) > 0$, then  $k(\bdx, \bdy) = 
\rho(\bdx-\bdy)\chi_{\mathcal{C}}(\bx-\by)$ satisfies all the conditions of the 
theorem. In this  case the kernel is not symmetric. However,  its symmetric 
part is $k_{s}(\bdx, \bdy) =\rho(\bdx-\bdy)\chi_{\mathcal{C}\cup 
(-\mathcal{C})}(\bx-\by)$, and the union $\mathcal{C}\cup (-\mathcal{C})$ is 
now a double cone with apex at the origin. The antisymmetric part of $k$ is 
given by $k_a(\bdx, \bdy) = 
\frac{1}{2}\left(\rho(\bdx-\bdy)\chi_{\mathcal{C}}(\bx-\by)-\rho(\bdx-\bdy)\chi_
{\mathcal{(-C)}}(\bx-\by)\right)$ and satisfies both conditions \eqref{C1} and 
\eqref{C2} as can easily be seen. 

Before we give other examples let us first prove a lemma that helps us compare 
function spaces. The result is an improvement of \cite[Lemma 
2.12]{Du-Zhou2011}, where the same result is shown for radial kernels that are 
supported on $\Lambda_{B_r} = B_r$.
\begin{lemma} [{\bf Fractional Korn inequality}]\label{lem:Korns-lemma}
Let $s\in (0, 1)$ and let $m({\bfs \xi})$ be an even function defined on 
$B_r({\bfs 0})$ with the property that $0<\alpha_1 \leq m({\bfs \xi}) \leq 
\alpha_{2} < \infty$ for some positive constants $\alpha_1$ and $\alpha_2$.  
For a given $\Lambda$ a double cone with apex at the origin and a given $r > 0$ 
define the kernel
 \[
 k_r(\bdx, \bdy) = \frac{m({\bdx-\bdy})  }{|\bdx - \bdy|^{d + 
2s}}\chi_{\Lambda_{B_{r}}}(\bx-\by).  
 \]
Then the function space $S(\mathbb{R}^{d};k_r)$ is precisely 
$H^{s}(\mathbb{R}^{d};\mathbb{R}^{d})$. Moreover, there exists a function 
$\beta(r)$ with the property that $\beta(r)\to 0$ as $r\to \infty$, and 
positive constants $C_1$, $C_2$ such that 
\begin{equation}\label{equi-korn}
C_{1}[{\bf u}, {\bf u}]_{S(\mathbb{R}^{d};k_r)}\leq |{\bf u}|_{H^{s}}^{2} \leq 
C_{2} [{\bf u}, {\bf u}]_{S(\mathbb{R}^{d};k_r)} + C_{2} \beta(r) \|{\bf 
u}\|_{L^{2}}^{2}. 
\end{equation}
If $\Lambda_{B_r}$ is replaced by $ \Lambda$, then $\beta$ can be taken to be 
the zero function.  
The constants $C_{1},$ $C_2$ and the function $\beta$ depend on $\alpha_i$, 
$\Lambda$, $d$ and $s$.
\end{lemma}
\begin{proof}
We prove the lemma using the Fourier transform. First let us introduce the 
following modification $$\tilde{m}({\bfs \xi}) = \left\{ \begin{split}
m({\bfs \xi})&\quad {\bfs \xi}\in \Lambda_{B_r}\\
\alpha_{1} & \quad {\bfs \xi} \in \Lambda \cap \complement B_{r}({\bfs 0}). 
\end{split}
\right.$$
Then $\tilde{m}$ is even, and  $\tilde{m}({\bfs \xi}) \geq \alpha_{1}$ for all 
${\bfs \xi} \in \Lambda$. 
Now, for ${\bf u}\in S(\mathbb{R}^{d};k_r)$
\[
\begin{split}
[{\bf u}, {\bf u}]_{S(\mathbb{R}^{d};k_r)} &+ 
\iintdm{\mathbb{R}^{d} \, \mathbb{R}^{d}}{ \alpha_1\frac{\left| ({\bf u}(\bdy) 
- {\bf u}(\bdx))\cdot \frac{(\bdy - \bdx)}{|\bdy-\bdx|} \right| ^{2}}{|\bdy - 
\bdx|^{d + 2s}}\chi_{\Lambda\cap \complement B_{r}({\bfs 0})}(\bx-\by)}{\bdy}{\bdx} \\
&=\iintdm{\mathbb{R}^{d} \, \mathbb{R}^{d}}{ \tilde{m}({\bdx-\bdy})\frac{ 
\left| ({\bf u}(\bdy) - {\bf u}(\bdx))\cdot \frac{(\bdy - \bdx)}{|\bdy-\bdx|} 
\right|^{2}}{|\bdy - \bdx|^{d + 2s}}\chi_{\Lambda}(\bx-\by)}{\bdy}{\bdx} \\
 &= \iintdm{\mathbb{R}^{d} \, \mathbb{R}^{d}}{ \tilde{m}({\bdx-\bdy})\frac{ 
\left| ({\bf u}(\bdy) - {\bf u}(\bdx))\cdot \frac{(\bdy - \bdx)}{|\bdy-\bdx|} 
\right|^{2}}{|\bdy - \bdx|^{d + 2s}}\chi_{\Lambda}(\bx-\by)}{\bdy}{\bdx}\\
  &= \intdm{\Lambda}{{\tilde{m}({\bf h})\over |\bf h|^{d + 2s}} \|\tau_{{\bf 
h}}{\bf u}\|^{2}_{L^{2}(\mathbb{R}^{d})}}{\bh}
\end{split}
\]
where $\tau_{{\bf h}}{\bf u}(\bdx) = ({\bf u}(\bdx + {\bf h}) - {\bf 
u}(\bdx))\cdot \frac{{\bf h}}{|{\bf h}|}$. Note that the Fourier transform of 
$\tau_{{\bf h}}{\bf u}(\bdx)$ is given by
\[
\mathcal{F}(\tau_{{\bf h}}{\bf u})(\xi) = (e^{\imath 2\pi {\bfs \xi}\cdot {\bf 
h}} -1)\mathcal{F}({\bf u})({\bfs \xi})\cdot  \frac{{\bf h}}{|{\bf h}|}. 
\]
Using Parseval's identity and after a simple calculation we see that
\[
\begin{split}
\|\tau_{{\bf h}}{\bf u}\|_{L^{2}(\mathbb{R}^{d})}^{2} &=2 \intdm{\mathbb{R}^{d}}{ 
 \left|\mathcal{F}({\bf u})({\bfs \xi}) \cdot \frac{{\bf h}}{|{\bf 
h}|}\right|^{2} (1 - \cos(2\pi {\bfs \xi}\cdot {\bf h}))}{{\bfs \xi}}. 
\end{split}
\]
Plugging the last expression in the above semi-norm and interchanging the 
integral we get that 
\[
\begin{split}
\intdm{\Lambda}{{\tilde{m}({\bf h})\over |\bf h|^{d + 2s}} \|\tau_{{\bf h}}{\bf 
u}\|^{2}_{L^{2}(\mathbb{R}^{d})}}{\bh}
&= 2  \intdm{\mathbb{R}^{d}}{ \left[ \intdm{\Lambda} {{\tilde{m}({\bf h})\over |\bf h|^{d 
+ 2s}}\left|\mathcal{F}({\bf u})({\bfs \xi}) \cdot \frac{{\bf h}}{|{\bf 
h}|}\right|^{2} (1 - \cos(2\pi {\bfs \xi}\cdot {\bf h}))}{\bh} \right]}{{\bfs \xi}} \\
&\geq 2\alpha_{1} \intdm{\mathbb{R}^{d}}{ \left[ \intdm{\Lambda}{{(1 - \cos(2\pi {\bfs 
\xi}\cdot {\bf h}))\over |\bf h|^{d + 2s}}\left|\mathcal{F}({\bf u})({\bfs 
\xi}) \cdot \frac{{\bf h}}{|{\bf h}|}\right|^{2}}{\bh} \right]}{{\bfs \xi}}\\
& = 2\alpha_{1} \intdm{\mathbb{R}^{d}}{ |{\bfs \xi}|^{2s} \left[ \intdm{\Lambda}{{(1 - 
\cos(2\pi {{\bfs \xi} \over{|{\bfs \xi}|} }\cdot {\bf h}))\over |\bf h|^{d + 
2s}}\left|\mathcal{F}({\bf u})({\bfs \xi}) \cdot \frac{{\bf h}}{|{\bf 
h}|}\right|^{2}}{\bh} \right] }{{\bfs \xi}},  
\end{split}
\]
where in the last step we have made a change of variables ${\bf h} \mapsto 
|{\bfs \xi}|{\bf h}$ and used the fact that $\Lambda $ remains invariant under 
scaling. Notice that the last inequality can be written as 
\[
\begin{split}
[{\bf u}, {\bf u}]_{S(\mathbb{R}^{d};k)} &+ 
\iintdm{\mathbb{R}^{d} \, \mathbb{R}^{d}}{ \alpha_1\frac{ \left| ({\bf u}(\bdy) 
- {\bf u}(\bdx))\cdot \frac{(\bdy - \bdx)}{|\bdy-\bdx|} \right|^{2}}{|\bdy - 
\bdx|^{d + 2s}}\chi_{\Lambda\cap \complement B_{r}({\bfs 0})}(\bx-\by)}{\bdy}{\bdx} \\
	&\geq  2\alpha_{1} \intdm{\mathbb{R}^{d}}{ |{\bfs \xi}|^{2s}|\mathcal{F}({\bf 
u})({\bfs \xi})|^{2} \Psi\left({\mathcal{F}({\bf u})({\bfs \xi})\over 
|\mathcal{F}({\bf u})({\bfs \xi})|}, {{\bfs \xi} \over{|{\bfs \xi}|} } \right) }{{\bfs \xi}},  
\end{split}
\]
where the map $\Psi: \mathbb{S}^{d-1}\times \mathbb{S}^{d-1} \to [0, \infty)$ 
is given by  
$
\Psi({\bfs \nu}, {\bfs \eta}) = \int_{\Lambda} {(1 - \cos(2\pi {\bfs \nu}\cdot 
{\bf h}))\over |\bf h|^{d + 2s}}\left|{\bfs \eta} \cdot \frac{{\bf h}}{|{\bf 
h}|}\right|^{2} \, \mathrm{d}{\bf h}. 
$
It is not difficult to see that $\Psi $ is a continuous positive function on 
the compact set $\mathbb{S}^{d-1}\times \mathbb{S}^{d-1}$, and therefore has a 
positive minimum, $\Psi_{min}$. As a consequence we have 
\[
\begin{split}
2\alpha_{1} \Psi_{min} |{\bf u}|_{H^{s}}^{2} &= 2\alpha_{1} 
\Psi_{min}\intdm{\mathbb{R}^{d}}{ |{\bfs \xi}|^{2s}|\mathcal{F}({\bf u})({\bfs 
\xi})|^{2}}{{\bfs \xi}}\\
&\leq  [{\bf u}, {\bf u}]_{S(\mathbb{R}^{d};k)} + 
\iintdm{\mathbb{R}^{d} \, \mathbb{R}^{d}}{ \alpha_1\frac{ \left| ({\bf u}(\bdy) - {\bf 
u}(\bdx))\cdot \frac{(\bdy - \bdx)}{|\bdy-\bdx|}\right|^{2}}{|\bdy - \bdx|^{d + 
2s}}\chi_{\Lambda\cap \complement B_{r}({\bfs 0})}(\bx-\by)}{\bdy}{\bdx}. 
 \end{split}
\]
Next, we estimate the second term on the right hand side of the above 
inequality. Again using the Fourier transform we have that 
\[
\begin{split}
\iintdm{\mathbb{R}^{d} \, \mathbb{R}^{d}}{ &\alpha_1\frac{ \left|({\bf u}(\bdy) 
- {\bf u}(\bdx))\cdot \frac{(\bdy - \bdx)}{|\bdy-\bdx|} \right|^{2}}{|\bdy - 
\bdx|^{d + 2s}}\chi_{\Lambda\cap \complement B_{r}({\bfs 0})}(\bx-\by)}{\bdy}{\bdx} \\
&= 2\alpha_{1} \intdm{\mathbb{R}^{d}}{ \intdm{\Lambda\cap \complement B_{r}({\bfs 
0})}{{(1 - \cos(2\pi {\bfs \xi}\cdot {\bf h}))\over |\bf h|^{d + 
2s}}\left|\mathcal{F}({\bf u})({\bfs \xi}) \cdot \frac{{\bf h}}{|{\bf 
h}|}\right|^{2}}{{\bf h}}}{{\bfs \xi}}\\
&\leq 2\alpha_1 \beta(r) \intdm{\mathbb{R}^{d}}{ |\mathcal{F}({\bf u})({\bfs 
\xi})|^2 }{{\bfs \xi}}, 
\end{split}
\]
where $\beta(r) = \int_{\Lambda\cap \complement B_{r}({\bfs 0})} {\mathrm{d}{\bf h}\over 
|{\bf h}|^{d + 2s}}\to 0,$ as $r\to \infty$ and depends only on $d, s, $ and 
$\Lambda$. We conclude that there exists a constant $C>0$ such that for every 
${\bf u}\in S(\mathbb{R}^{d};k)$ we have  
$
|{\bf u}|_{H^{s}}^{2} \leq C [{\bf u}, {\bf u}]_{S(\mathbb{R}^{d};k)} + C 
\beta(r) \|{\bf u}\|_{L^{2}}^{2}. 
$
The bound
\[
[{\bf u}, {\bf u}]_{S(\mathbb{R}^{d};k)} \leq 2 \alpha_2  \Psi_{max} |{\bf 
u}|_{H^{s}}^{2}
\]
can be proved in a similar fashion. 
\end{proof}

\noindent Let us now continue discussing examples of kernels that may satisfy our 
well-posedness result.  

\vspace{0.01in}

\textit{Example 3:} Let $k_r$ be as in \autoref{lem:Korns-lemma}. 
Since  the kernel is symmetric, to check the applicability 
of \autoref{thm:SolutionofD0} for this kernel, we need to verify only the 
Poincar\'e-Korn type inequality.  But this follows from 
\autoref{lem:SufficientPoincareLemma} by taking $\rho({\bfs \xi})= |{\bfs 
\xi}|^2 k_r({\bfs \xi})$, for any $r>0$.  
By above lemma, the space $S(\mathbb{R}^{d};k_r)$ is in fact  
$H^{s}(\mathbb{R}^{d};\mathbb{R}^{d})$.   

\textit{Example 4:} Another nontrivial non-symmetric kernel given in \cite{Kassmann} is 
the following. For $s\in (0,1)$, fix $\alpha\in \left( 0, {s\over 2} \right)$. 
Let $\Lambda$ be a double cone  with apex at the origin. Given the cone 
$\mathcal{C} = \left\lbrace {\bf h}\in B_{1}({\bfs 0}): \frac{{\bfs h}}{{|{\bf 
h}|} } \in \mathcal{J} \right\rbrace$, and $\mathcal{J}$
 is a nontrivial open subset of the unit sphere $\mathbb{S}^{d-1}, $ such that 
$-\mathcal{J} \neq \mathcal{J}$, let us consider the kernel 
 \[
 k(\bdx, \bdy) = \frac{1}{|\bdx-\bdy|^{d+2s}} \chi_{\Lambda}(\by-\bx) + 
\frac{1}{|\bdx-\bdy|^{d+2\alpha}} \chi_{\mathcal{C}}(\by-\bx). 
 \]
 Then the symmetric and antisymmetric part of $k$ are given by 
 \[
 \begin{split}
 k_s(\bx, \by) &= \frac{1}{|\bdx-\bdy|^{d+2s}} \chi_{\Lambda}(\by-\bx) + 
\frac{1}{2}  \frac{1}{|\bdx-\bdy|^{d+2\alpha}} \chi_{\mathcal{C}\cup 
(-\mathcal{C})}(\by-\bx)\\
  k_a(\bx, \by) &=  \frac{1}{2}\frac{1}{|\bdx-\bdy|^{d+2\alpha}} 
\chi_{\mathcal{C}}(\by-\bx) - \frac{1}{2}  \frac{1}{|\bdx-\bdy|^{d+2\alpha}} 
\chi_{(-\mathcal{C})}(\by-\bx).
 \end{split}
 \]
 Conditions \eqref{SM} and \eqref{C1}-\eqref{C2} can be shown as in 
\cite{Kassmann}. Let us show \eqref{k1}-\eqref{k2} with $\tilde{k}(\bx, \by) = 
\frac{1}{|\bx-\by|^{d+2s}}$. Again, \eqref{k2} is shown in \cite{Kassmann} 
where the constant $A_2$ depends on $\mathcal{C}$ and $s-{2\alpha}$, but to 
show \eqref{k1} we use the fact that $k_{s}(\bx, \bdy) \geq 
\frac{1}{|\bdx-\bdy|^{d+2s}} \chi_{\Lambda}(\by-\bx)$. Indeed, 
using \autoref{lem:Korns-lemma} and the remark following it, there exist 
constants $c_1, c_2 > 0$ such that 
\[
\begin{split}
\iintdm{\bbR^d\,  \bbR^d}{\tilde{k}(\bdx, \bdy)\left(({\bf u}(\by) - 
{\bf u}(\bdx))\cdot{\bdy-\bdx \over |\bdy-\bdx|}\right)^2}{\bdx}{\bdy} &\leq 
c_1|{\bf u}|^{2}_{H^{s}} \\
& \leq c_{2} \, \iintdm{\bbR^d\,  \bbR^d} 
{\frac{\chi_{\Lambda}(\by-\bx)}{|\bdx-\bdy|^{d + 2s}} 
|\cD(\bu)(\bx,\by)|^2} {\bdx} {\bdy} \\
&\leq  2c_2 [{\bf u}, {\bf u}] _{S(\mathbb{R}^{d};k)}\,.
\end{split}
 \]
The Poincar\'e-Korn inequality $(PK)$ now follows from the standard Fractional 
Poincar\'e inequality, because the function space $S(\mathbb{R}^{d};k)$ 
coincides with $H^{s}(\mathbb{R}^{d};\mathbb{R}^{d})$, and because by 
\autoref{lem:Korns-lemma} \\ 
$$
|{\bf u}|^{2}_{H^{s}} \leq C [{\bf u}, {\bf u}] 
_{S(\mathbb{R}^{d};k)}.
$$ 

\subsubsection{Variants of the Dirichlet problem}

As indicated earlier in the proof of \autoref{thm:SolutionofD0}, conditions 
\eqref{C1}-\eqref{C2} on the kernel $k$ are used to show the positive 
semi-definiteness of the bilinear form on $S(\mathbb{R}^{d}; k)$.  There are 
however kernels for which either these conditions are not true or difficult to 
verify. For this class of kernels, well-posedness of the Dirichlet problem 
corresponding to the addition of a positive multiple of the identity operator 
can be obtained.

\begin{proposition}\label{prop:withbeta}
Let $\Omega \subset \bbR^d$ be open, bounded. Let $k$ satisfy \eqref{SM} and 
\eqref{k1}-\eqref{k2}. Assume also that $(PK)$ holds. Then there exists 
$\beta_{0} > 0$ such that for any $\beta > \beta_{0}$ and any ${\bf f} \in 
S^{\ast}_{\Omega}(\bbR^d;k)$, there exists a unique solution $\bu \in 
\HOmegaRd$ to 
\begin{equation}
\begin{cases}
\mathbb{ L}\bu  + \beta {\bf u}= {\bf f} & \quad \text{in } \Omega\,, \\
\bu = {\bfs 0} & \quad \text{on } \complement \Omega\,.
\end{cases}
\end{equation}
Moreover, there exists a constant $c>0$ independent of ${\bf f}$ such that 
 \[
 [{\bf u}, {\bf u}]_{\HOmegaRd} \leq c \|{\bf 
f}\|_{S^{*}_{\Omega}(\mathbb{R}^{d}; k)}. 
 \]
\end{proposition}
\begin{proof}
The proof follows from standard arguments once Gårding-type estimates are 
established. To that end, we show that there is a constant $\gamma = 
\gamma(A_1,A_2) > 0$ such that
\begin{equation}\label{GardingsInequality}
\cF^k(\bu,\bu) \geq \frac{1}{4} \Vert \bu \Vert^2_{\HRd} - \gamma \Vert \bu 
\Vert^2_{L^2(\bbR^d)} \quad \text{for all }\bu \in \HRd.
\end{equation}
To prove this, let $\bu \in \HRd$.  From \eqref{k1}-\eqref{k2} and by Young's 
inequality,
\begin{align*}
\cF^k(\bu,\bu) &\geq \frac{1}{2}\iintdm{\bbR^d \, \bbR^d}{k_s(\bx,\by) 
(\mathcal{D}({\bf u})(\bdx, \bdy))^2}{\by}{\bx}  - \iintdm{\bbR^d \, 
\bbR^d}{k_a(\bx,\by) |\mathcal{D}({\bf u})(\bdx, \bdy)| \left| \bu(\bx) \cdot 
\frac{\bx-\by}{|\bx-\by|} \right|}{\by}{\bx} \\
&\geq \frac{1}{2}[\bu,\bu]_{\HRd} - \iintdm{\bbR^d \, \bbR^d}{|\mathcal{D}({\bf 
u})(\bdx, \bdy)| \tilde{k}^{1/2}(\bx,\by) 
|\bu(\bx)|k_a(\bx,\by)\tilde{k}^{-1/2}(\bx,\by)}{\by}{\bx} \\
&\geq \frac{1}{2}[\bu,\bu]_{\HRd} - \iintdm{\bbR^d \, \bbR^d}{\left( \veps 
|\mathcal{D}({\bf u})(\bdx, \bdy)|^2 \tilde{k}(\bx,\by) + 
\frac{1}{4\veps}|\bu(\bx)|^2 k_a^2(\bx,\by) 
\tilde{k}^{-1}(\bx,\by)\right)}{\by}{\bx} \\
&\geq \frac{1}{4}[\bu,\bu]_{\HRd} - C(\epsilon)\Vert \bu \Vert^2_{L^2(\bbR^d)} 
\\
&\geq \frac{1}{4}\Vert \bu \Vert^2_{\HRd} - \gamma \Vert \bu 
\Vert^2_{L^2(\bbR^d)},
\end{align*}
if $\veps$ is chosen sufficiently small such that $A_1 \epsilon < 1/4$ and then 
$\gamma = \gamma(A_1, A_2)$ chosen sufficiently large.
\end{proof}
We next discuss an example of a nontrivial  kernel that satisfies all the 
conditions of the proposition. The example is taken from \cite{Schilling, 
Fukushima, Kassmann} and discussed in detail there. For two given positive 
numbers $0<\alpha_{1}\leq \alpha_{2}<2$, let $\alpha:\mathbb{R}^{d}\to 
[\alpha_{1}, \alpha_{2}] $ be a continuous function, with its modulus of 
continuity $\omega[\alpha]$ satisfying 
$\int_{0}^{1}\frac{(\omega[\alpha](r)\ln(r))^{2}}{r^{1 + \alpha_2}} \mathrm{d}r < 
\infty$.  We introduce the non-symmetric kernel 
\[
k(\bdx, \bdy) = \frac{b(\bdx)}{|\bdx-\bdy|^{d + \alpha(\bdx)}},
\]
where $b(\bdx)$ is a continuous function bounded from below and  above by 
positive numbers and satisfying the inequality $|b(\bdx)-b(\bdy)|\leq 
c|\alpha(\bdx) - \alpha(\bdy)|$ for some  $c>0$ provided $|\bdx - \bdy| < 1$. 
To see if \autoref{prop:withbeta} applies to this kernel, we need to verify 
\eqref{SM},  \eqref{k1}-\eqref{k2} and $(PK)$. It has been shown in 
\cite{Schilling} that this kernel satisfies \eqref{SM} and 
\eqref{k1}-\eqref{k2}, with $\tilde{k}$ taken to be the symmetric part $k_{s}$ 
of $k$.  What remains is the show the  Poincar\'e-Korn inequality $(PK)$ holds 
for $k$. But this follows from \autoref{lem:SufficientPoincareLemma} and the 
fact that $k_{s}(\bdx, \bdy) \geq \frac{b_{min}}{|\bdx-\bdy|^{d + \alpha_{1}}}$ 
when $|\bdx - \bdy| < 1$. 

We would like to mention that in \cite{Kassmann} for the modified kernel 
$k'(\bx, \by) = \chi_{B_{R}({\bfs 0})} (\by-\bx) k(\bdx, \bdy)$, for $1\ll R$,  
the Dirichlet problem for scalar equations is shown to be well-posed even for 
$\beta = 0$, see \cite[Theorem 4.4]{Kassmann}. This was possible using the 
Fredholm Alternative theorem via the application of the weak maximum principle 
that is used to prove uniqueness of the solution to the Dirichlet problem with 
zero right hand side.   Following the argument in \cite{Kassmann}, one can 
write a Fredholm Alternative theorem for the Dirichlet problem $(D_0)$ of the 
system of nonlocal equations. However, since we are dealing with system of 
equations a maximum principle is not applicable and we are unable to show 
uniqueness of the solution of the linear system of equations $(D_0)$.   The 
uniqueness of the zero solution $(D_0)$ corresponding to ${\bf f}=0$ under the 
assumption of \autoref{prop:withbeta} or even the stronger assumption on $k$ 
given in \cite[
Theorem 4.4]{Kassmann} remains an open problem.  

We end this section by noting that well-posedness of the Dirichlet problem with 
nonzero complementary data can also be proved. To that end, again following the 
set up in \cite{Kassmann}, let us introduce the function space 
\[
V(D;k) = \left\{ \bv:\bbR^d \to \bbR^d \,: \, \bv|_{D} \in 
L^2(D;\mathbb{R}^{d}), (\bv(\bx)-\bv(\by)) \cdot 
\frac{(\bx-\by)}{|\bx-\by|}k_s^{1/2}(\bx,\by) \in L^2(D \times \bbR^d) 
\right\}\,.
\]
The mapping $[\bu, \bv]_{V(D;k)}$ given by 
\[
[ \bu,\bv ]_{V(D;k)} := \iintdm{D \, \bbR^d}{k_s(\bx,\by) \diffqbu 
\diffqbv}{\by}{\bx}
\]
defines a bilinear form on $V(D;k)$. In the event  that $D = \mathbb{R}^{d}$, 
it is clear that $V(\bbR^d;k)=\HRd, $ where $\HRd$ is defined in the 
introduction.   
For a given ${\bf g} \in V(\Omega;k)$, we say $\bu \in V(\Omega;k)$ is called a 
solution of
\begin{equation}
\mathbb{ L}\bu = {\bf f}  \quad \text{in } \Omega, \quad\quad
\bu = \bg \quad \text{on } \complement \Omega,
\tag{$D$}
\end{equation}
if $\bu-\bg \in \HOmegaRd$ and \eqref{weak-pairing} holds. 

We now state the well-posedness of the Dirichlet Problem. We omit the proof here 
as it can be done following the argument in  \cite{Kassmann}.  

\begin{theorem}\label{thm:WellPosednessofDLaxMilgram}
Let $\Omega \subset \bbR^d$ be open, bounded. Let $k$ be a kernel that 
satisfies \eqref{SM}, \eqref{C1}-\eqref{C2}, and  $(PK)$. Assume further that 
there exists a $\tilde{k}$ such that for all ${\bf u} \in V(\Omega;k)$
\begin{align}\label{eq:k1-omega}
\begin{split}
&\iintdm{\Omega \, \bbR^d}{\diffqu^2 \tilde{k}(\bdx,\bdy)}{\bdy}{\bdx} \\
&\quad \leq A_1 
\iintdm{\Omega \, \bbR^d}{\diffqu^2 k_s(\bdx,\bdy)}{\bdy}{\bdx}
\end{split}
\end{align}
and such that \eqref{k2} holds for this $\tilde{k}$. Then $(D)$ has a unique 
solution ${\bf u} \in V(\Omega;k)$, with
\begin{equation}\label{SolutionEstimate}
[{\bf u},{\bf u}]_{V(\Omega;k)} \leq C \left( \Vert f 
\Vert^2_{S^*_{\Omega}(\bbR^d;k)} + [g,g]_{V(\Omega;k)} \right)\,,
\end{equation}
where $C = C(C_P,A_1,A_2) > 0$.
\end{theorem}

\begin{remark}
Condition \eqref{eq:k1-omega} obviously holds if one chooses 
$\tilde{k}(\bdx,\bdy) = k_s(\bdx,\bdy)$. The integration allows for more 
flexibility here, see \cite{Kassmann} for examples. Note that 
\autoref{thm:WellPosednessofDLaxMilgram} opens up an interesting question 
concerning data on $\complement \Omega$. The result requires ${\bf g} \in 
V(\Omega;k)$, i.e., the data is given in all of $\mathbb{R}^d$. This condition 
is similar to the condition $g \in H^1(\Omega)$ when searching for a 
solution $v$ solving some partial differential equation of second order in 
$\Omega$ with $v-g \in H^1_0(\Omega)$. From point of view of applications it is 
desirable to prescribe $g$ only on the complement $\complement \Omega$ and to 
have some extension theorem. Such results are nowadays standard for classical 
Sobolev function spaces. They put into relation the trace space 
$H^{\frac{1}{2}}(\Omega)$ with $H^1(\Omega)$. For spaces with fractional order 
of derivative, a similar relation has been addressed in 
\cite{DyKa16}.

\end{remark}

\section{Interior Regularity of Solutions}\label{sec:regularity}
\subsection{Setup and main results}
We now turn to the question of regularity of solutions. We want to answer the 
following question: if the data ${\bf f}$ are in $ 
L^{p}(\Omega;\mathbb{R}^{d})$, what is the optimal space for the weak 
solution ${\bf u}$ of the Dirichlet problem of the system of nonlocal 
equations ($D_0$)? From the existence result proved in the previous 
section, if ${\bf f}\in S^{*}_{\Omega}(\mathbb{R}^{d}; k)$ then ${\bf u}\in 
S_{\Omega}(\mathbb{R}^{d}; k)$, which is the largest space to which the 
solution can belong. This space is, in general, not the optimal space.  
Moreover, for general kernels $k$ there is no good characterization of the space 
or other finer subspaces in which the solution may live. With this in mind, in 
this section we give a partial result concerning regularity of solutions. The 
result applies to systems of equations with leading operator $\mathbb{L}$ 
defined using an even function comparable with the fractional kernel. To be 
precise, let $s\in (0, 1)$ and $m({\bfs \xi})$ be an even 
function with the property that $0<\alpha_1 \leq m({\bfs \xi}) \leq \alpha_{2} < 
\infty$ for some positive constants $\alpha_1$ and $\alpha_2$.  For a given 
$\Lambda$, a double cone with apex at the origin, and $0< r \leq \infty$ we 
consider translation-invariant kernels that may be supported on $\Lambda$: 
 \begin{equation}\label{form-k}
 k_r(\bdx-\bdy) = \frac{m({\bdx-\bdy})  }{|\bdx - \bdy|^{d + 
2s}}\chi_{\Lambda_{B_r}}(\bx-\by).  
 \end{equation}
For kernels of this form we have shown in \autoref{lem:Korns-lemma} that $ 
\HRd= H^{s}(\mathbb{R}^{d};\mathbb{R}^{d})$.  Noting that 
$H_{\Omega}(\mathbb{R}^{d};k) = 
L^{2}_{\Omega}(\mathbb{R}^{d};\mathbb{R}^{d})\cap \HRd$, we have that 
$S_{\Omega}(\mathbb{R}^{d};k) = L^{2}_{\Omega}(\mathbb{R}^{d};\mathbb{R}^{d}) 
\cap H^{s}(\mathbb{R}^{d};\mathbb{R}^{d})$. We denote the latter set by 
$H^{s}_{\Omega}(\mathbb{R}^{d};\mathbb{R}^{d})$. We also denote 
\[
H^{2s}_{loc}(\Omega;\mathbb{R}^{d}) = \{{\bf u}\in L^{2}(\Omega; 
\mathbb{R}^{d}): \eta {\bf u} \in H^{2s}(\mathbb{R}^{d};\mathbb{R}^{d}), \quad 
\forall \eta \in C_{c}^{\infty}(\Omega)\}. \]
We also need the following potential spaces. If we denote ${\bf u} \in 
\mathcal{S}'$, the space of tempered distributions, then 
$$
\cL^{s,p}(\bbR^d) := \{ {\bf u} \in \mathcal{S}':  \mathcal{F}^{-1}[(1+|{\bfs 
\xi}|^{2})^{s/2}\mathcal{F}({\bf u})]\in  L^p(\bbR^d;\mathbb{R}^{d}) \}, 
$$ 
where $1 < p < \infty$ and $s \geq 0$. The space is equivalent to $\{{\bf u}\in 
L^{p}: (-\Delta)^{s/2} {\bf u}\in  L^{p}\}$, where $(-\Delta)^{s/2}$ is the
standard 
fractional Laplacian operator applied componentwise.  We also denote 
\[
\mathcal{L}^{s,p}_{loc}(\Omega;\mathbb{R}^{d}) := \{{\bf u}\in L^{p}(\Omega; 
\mathbb{R}^{d}): \eta {\bf u} \in 
\mathcal{L}^{s,p}(\mathbb{R}^{d};\mathbb{R}^{d}), \quad \forall \eta \in 
C_{c}^{\infty}(\Omega)\}. \]
When $p=2$, $\cL^{s,2}(\bbR^d) = H^{s}(\mathbb{R}^{d};\mathbb{R}^{d})$. 

The following theorem contains one of the results of this section. 

\begin{theorem}\label{thm:InteriorRegularity-L2MainTheorem}
Assume that $k_r$ has the form \eqref{form-k}. 
Let $\Omega \subset \bbR^d$ be a bounded open set, ${\bf f} \in 
L^2_{\Omega}(\bbR^d;\mathbb{R}^{d})$ and  ${\bf u} \in 
S_{\Omega}(\mathbb{R}^{d};k)$ be the unique weak solution to the system 
\begin{equation}\label{eq-InteriorRegularity-DirichletProblem}
\begin{cases}
	\mathbb{L}{\bf u} = {\bf f}  & \text{ on } \Omega \\
	{\bf u} = 0 & \text{ on } \complement \Omega.
\end{cases}
\end{equation}
Then ${\bf u} \in H^{2s}_{loc}(\Omega;\mathbb{R}^{d})$. Moreover, for any $\eta 
\in C_c^{\infty}(\Omega)$, there exists a constant $C>0$ depending on $\eta$ 
such that 
\[
|\eta{\bf u}|_{H^{2s}} \leq C \|{\bf f}\|_{L^{2}}. 
\]
\end{theorem}

Our second regularity result corresponds to the case when ${\bf f} \in 
L^p_{\Omega}(\bbR^d;\mathbb{R}^{d})$ for $p \geq 2$. For this result, we only 
study  $\mathbb{L}$ corresponding to $m = 1$ and $\Lambda  = \mathbb{R}^{d}$. 
That is,  $k$ is the standard fractional kernel given by  $k(x,y) = 
|x-y|^{-d-2s}$. To separate this special operator from generic ones, we 
introduce the notation $(-\mathring{\mathbf{\Delta}})^{s}$ to denote the matrix 
operator. That is, 
\[
(-\mathring{\mathbf{\Delta}})^{s}{\bf u} = \emph{p.v.} \intdm{\mathbb{R}^{d}}{ 
\shapetensorbxy {\bdu(\bdx)-\bdu(\bdy)\over { |\bx-\by|^{d+2s}}}}{\bdy}. 
\]
Note that, here and above, we omit a normalizing constant depending on 
$s$ and $d$ in the intro-differential representation 
of the fractional Laplace-type operator $(-\mathring{\mathbf{\Delta}})^{s}$. We 
do not study the limit cases  $s\nearrow 1$ or $s\searrow 0$.

\begin{theorem}\label{thm:int-for-p}
Let $p \in [2,\infty)$,  ${\bf f} \in L^p_{\Omega}(\bbR^d)$ and  ${\bf u} \in 
S_{\Omega}(\bbR^d;k)$ be the unique weak solution to the Dirichlet problem 
\eqref{eq-InteriorRegularity-DirichletProblem} with $\mathbb{L}$ replaced by 
$(-\mathring{\mathbf{\Delta}})^{s}$.  
Then ${\bf u} \in \mathcal{L}^{2s,p}_{loc}(\Omega)$ provided
\begin{enumerate}
\item[a)] $2\leq p\leq 2^{*_{s}}$, where $2^{{*}_s} := {2d\over d-2s}$; 
\item [ ] or 
\item [b)] $p > 2^{*_{s}}$ and ${\bf u}\in L^{p}(\Omega;\mathbb{R}^{d})$. 
\end{enumerate}
\end{theorem}
As we describe in the introduction, to prove 
\autoref{thm:InteriorRegularity-L2MainTheorem} and \autoref{thm:int-for-p} we 
follow an argument used in \cite{Biccari-Warma}, where a similar but more 
general result is proved for the Dirichlet problem for fractional Laplacian 
equation when the right hand side comes from $L^p$ for any $1<p<\infty$.   
The argument relies on an optimal regularity result for weak solutions of the 
same system posed on the entire space. Multiplying the weak solution of the 
Dirichlet problem by a cutoff function, the product becomes a weak solution 
of a system of equations posed on $\mathbb{R}^{d}$ with a perturbed right hand 
side. The task is then to show that the perturbed force term lives in the same 
space as the original right hand side function.  In implementing the strategy 
of  \cite{Biccari-Warma}  to our case, although the cutoff function argument 
remains the same, we have to demonstrate the optimal regularity result for weak 
solutions of the strongly coupled system in $\mathbb{R}^{d}$.  
For strong solutions of nonlocal equations defined on $\mathbb{R}^{d}$, optimal regularity is obtain in \cite{DoKi12}. 

We should mention that, for the scalar case,  the result \cite[Theorem 1.4]{Biccari-Warma} 
does not require ${\bf u}$ be in $L^{p}(\Omega)$ for large $p$ as we do in 
 \autoref{thm:int-for-p}.  Roughly speaking, they 
prove that a solution to the Dirichlet problem of the fractional Laplacian with 
right hand side in $L^{p}$ must also be in $L^{p}$, see \cite[Lemma 
2.5]{Biccari-Warma}. A similar Calderon-Zygmund type estimate is also proved in 
\cite[Theorem 16]{Peral}. Unfortunately we are unable to extend their  proof to 
the vector-valued case because the argument in  \cite{Biccari-Warma} relies on 
a monotonicity property of an associated semigroup and in the case of 
\cite{Peral} it 
uses a Moser-type argument where a function that is a power of the solution is 
used as a test function. which we cannot do for systems.  

\subsection{Interior $H^{2s}$ Regularity for the Dirichlet Problem of the 
System of Equations}
Now we turn to the main point. 
We recall that for a given ${\bf f}\in L^{2}(\mathbb{R}^{d};\mathbb{R}^{d})$, 
we say that ${\bf u}\in S(\mathbb{R}^{d};k)$ is a weak solution to $
\mathbb{L}{\bf u} = {\bf f}
$ in $\mathbb{R}^{d}$ 
if 
\begin{equation}\label{weak-whole-space}
\langle \mathbb{L}{\bf u}, {\bfs \psi} \rangle := 
\iintdm{\mathbb{R}^{d} \, \mathbb{R}^{d}}{{\mathcal{D} ({\bf u})(\bdx, \bdy) 
\mathcal{D} ({\bfs \psi})(\bdx, \bdy)} k(\bdx-\bdy)}{\bx}{\by} = 
\intdm{\mathbb{R}^{d}}{{\bf f}(\bdx)\cdot {\bfs \psi}(\bdx)}{\bdx}, \quad \forall 
{\bfs \psi}\in C^{\infty}_{c}(\mathbb{R}^{d};\mathbb{R}^{d}). 
\end{equation}
The following lemma gives an optimal regularity result for weak solutions of 
the system. 
\begin{lemma}\label{lem:lma-L2RegularityonRd}
Assume that $k_r$ has the form \eqref{form-k}. Suppose that ${\bf f}\in 
L^{2}(\mathbb{R}^{d};\mathbb{R}^{d})$. Let ${\bf u}\in 
H^{s}(\mathbb{R}^{d};\mathbb{R}^{d})$ be a weak solution to the system of 
nonlocal equations 
\[
\mathbb{L}{\bf u} = {\bf f}, \quad\text{in $\mathbb{R}^{d}$}.
\]
Then ${\bf u} \in H^{2s}(\bbR^d;\mathbb{R}^{d})$, and $|{\bf u}|_{H^{2s}} \leq 
c \left(\|{\bf f}\|_{L^{2}} + \|{\bf u}\|_{L^{2}}\right)$ for some constant c 
depending only on $r$, $s$, $d$, and $\Lambda$.  
\end{lemma}

\begin{remark}
Note that \cite{coz17} establishes very similar regularity results. 
\end{remark}

\begin{proof}
Let ${\bfs \psi}\in C^{\infty}_{c}(\mathbb{R}^{d};\mathbb{R}^{d})$. Then 
iterating the integral in \eqref{weak-whole-space} and changing variables  we 
have that 
\[
\intdm{\mathbb{R}^{d}}{k_r({\bf h})\intdm{\mathbb{R}^{d}}{ \left(({\bf u}(\bdx + 
{\bf h})- {\bf u}(\bdx)) \cdot \frac{{\bf h}}{|{\bf h}|} \right) \left(({\bfs 
\psi}(\bdx + {\bf h})- {\bfs \psi}(\bdx)) \cdot \frac{{\bf h}}{|{\bf h}|} 
\right) }{\bdx}}{\bh} = \intdm{\mathbb{R}^{d}}{{\bf f}(\bdx)\cdot {\bfs 
\psi}(\bdx)}{\bdx}.
\]
We apply the Fourier transform and Plancherel theorem to rewrite the above 
integral in the frequency space as 
\[
\intdm{\mathbb{R}^{d}}{ (\mathbb{M}_r({\bfs \xi}) \hat{{\bf u}} ({\bfs \xi})) 
\cdot \hat{{\bfs \psi}} ({\bfs \xi}) }{{\bfs \xi}} = 
\intdm{\mathbb{R}^{d}}{\hat{{\bf f}}({\bfs \xi})\cdot \hat{{\bfs \psi}}({\bfs 
\xi})}{{\bfs \xi}}, 
\]
where $\mathbb{M}_r({\bfs \xi})$ is the matrix of Fourier symbols given by 
\begin{equation*}
\mathbb{M}_r({\bfs \xi}) = \intdm{\mathbb{R}^{d}}{ k_r({\bf h}) (e^{\imath 2\pi 
{\bfs \xi} \cdot {\bf h}}  -1 )^{2} \frac{{\bf h}}{|{\bf h}|} \otimes 
\frac{{\bf h}}{|{\bf h}|}}{{\bf h}}.
\end{equation*}
We now use the density of the Fourier transform of 
$C_c^{\infty}(\mathbb{R}^{d};\mathbb{R}^{d})$ in $L^{2}(\mathbb{R}^{d}; 
\mathbb{R}^{d})$ to conclude that   
\begin{equation} 
\label{symbol-weak}
\mathbb{M}_r({\bfs \xi}) \hat{{\bf u}} ({\bfs \xi}) = \hat{{\bf f}}({\bfs 
\xi}), \quad \text{almost everywhere in $\mathbb{R}^{d}$.}
\end{equation}
Let us write $ k_r = k - \tilde{k}_r$ where $ k(\bz) = \frac{m({\bz})  
}{|\bz|^{d + 2s}}\chi_{\Lambda}(\bz).$ Notice that $\tilde{k}_r$ is supported 
outside of the ball $B_r$. If we denote the matrix of symbols by $\bbM$ and 
$\tilde{\bbM}_r$, we have that 
\begin{equation}\label{formula-for-matrix-symbol}
\mathbb{M}({\bfs \xi}) \hat{{\bf u}} ({\bfs \xi}) = \tilde{\bbM}_r({\bfs \xi}) 
\hat{{\bf u}} ({\bfs \xi})  + \hat{{\bf f}}({\bfs \xi}), \quad \text{almost 
everywhere in $\mathbb{R}^{d}$.}
\end{equation}
To estimate the relevant norms of ${\bf u}$, let us first estimate the 
eigenvalues of the matrix $\mathbb{M}({\bfs \xi})$. To that end, for any ${\bfs 
\eta}\in \mathbb{S}^{d-1}$, noting the form of $k$ we have that 
\[
\begin{split}
(\mathbb{M}({\bfs \xi}){\bfs \eta})\cdot{\bfs \eta}  =  \intdm{\mathbb{R}^{d}}{
k({\bf h}) (e^{\imath 2\pi {\bfs \xi} \cdot {\bf h}}  -1 )^{2}\left|{\bfs \eta} 
\cdot {{\bf h}\over |{\bf h}|}\right|^{2}}{\bh}
	&\geq  2\alpha_{1} 
\intdm{\Lambda}{ \frac{(1- \cos(2\pi {\bfs \xi}\cdot {\bf h}))}{|{\bf 
h}|^{d+2s}}\left|{\bfs \eta} \cdot {{\bf h}\over |{\bf h}|}\right|^{2}}{{\bf h}}
\\
&\geq 2\alpha_{1} \Psi_{min} |{\bfs \xi}|^{2s}\,,
\end{split}
\]
where the last inequality is from the proof of \autoref{lem:Korns-lemma}.  As 
a consequence the eigenvalues of the matrix function $|{\bfs \xi}|^{-2s} 
\mathbb{M}({\bfs \xi})$ are uniformly bounded from below by a positive number. 
We also note that since $\mathbb{M}({\bfs \xi})$ is symmetric and positive 
definite for each ${\bfs \xi}$, the eigenvalues of the square of 
$\mathbb{M}({\bfs \xi})$ are precisely the squares of the eigenvalues of 
$\mathbb{M}({\bfs \xi})$. It then follows that for any vector ${\bf w}$
\[
|\mathbb{M}({\bfs \xi}){\bf w}|^{2}  = \mathbb{M}({\bfs \xi}){\bf w}\cdot 
(\mathbb{M}({\bfs \xi}){\bf w}) = \mathbb{M}({\bfs \xi})\mathbb{M}({\bfs 
\xi}){\bf w}\cdot {\bf w} = |{\bf w}|^2\min_{\beta}\{\beta({\bfs \xi})^{2} \}\,,
\] 
where the minimum is taken over the eigenvalues $\beta({\bfs \xi})$ of 
$\mathbb{M}({\bfs \xi})$.  We conclude that there exists a positive number 
$\alpha_{0}$ that depends only on $\alpha_{1}, s, \Lambda$, such that for all 
vectors ${\bfs \xi}$ and ${\bf w}$ in $\mathbb{R}^{d}$ we have that 
\[
|\mathbb{M}({\bfs \xi}){\bf w}|^{2} \geq \alpha_{0}(|{\bfs \xi}|^{2s}|{\bf 
w}|)^{2}\,.
\]
We now easily see from \eqref{symbol-weak} that 
\[
\begin{split}
\Vnorm{\Dss {\bf u}}_{L^2(\bbR^d)}  = \||{\bfs \xi}|^{2s}\hat{{\bf 
u}}\|_{L^{2}}\leq \alpha_{0}  \|\mathbb{M}({\bfs \xi})\hat{{\bf u}}\|_{L^{2}}  
&= \alpha_{0}\left(\|\hat{\bf f}\|_{L^2(\bbR^d)} + + \|\tilde{\bbM}_r({\bfs 
\xi}) \hat{{\bf u}} ({\bfs \xi}) \|_{L^{2}}^2\right)\\
&\leq \alpha_0 \left(\|\hat{\bf f}\|_{L^2(\bbR^d)} + \beta(r) \|\hat{{\bf u}} 
\|_{L^{2}}^2\right)\,,
\end{split}
\]
where in the last inequality  $\beta(r)$ is from \autoref{lem:Korns-lemma}. 
Thus, since we already know that ${\bf u}\in 
L^{2}(\mathbb{R}^{d};\mathbb{R}^{d})$,  
we get that ${\bf u} \in \cL^{2s,2}$.
\end{proof}

\begin{proof}[Proof of \autoref{thm:InteriorRegularity-L2MainTheorem}]

Let $\Omega_{1}\Subset \Omega_2 \Subset \Omega$. Let $\eta \in 
C^{\infty}_c(\Omega_{2})$ be a real-valued function such that
\[
\eta(\bx) \equiv 1, \,\,\bx \in \Omega_{1},\quad \eta(\bx) \in [0,1], \,\bx \in 
\Omega_{2} \setminus \Omega_{1},\quad\text{and}\, \eta(\bx) = 0,  \,\,\bx \in 
\bbR^d \setminus \Omega_{2}.
\]

Let ${\bf f} \in L^2_{\Omega}(\bbR^d;\mathbb{R}^{d})$ and let ${\bf u} \in 
H_{\Omega}(\bbR^d;k)$ be the unique weak solution to the Dirichlet problem 
\eqref{eq-InteriorRegularity-DirichletProblem}. Notice that because of the form 
of $k$, ${\bf u}$ is in fact in 
$H_{\Omega}^{s}(\mathbb{R}^{d};\mathbb{R}^{d})$.  Now, it is clear that the 
function $\eta {\bf u} \in H_{\Omega}^{s}(\mathbb{R}^{d};\mathbb{R}^{d})$. 
Using the identity 
\[
\mathcal{D}({\bf u}\eta)(\bdy, \bdx) = (\eta(\bdx)-\eta(\bdy)) {\bf u}(\bdx) 
\cdot {\bdy-\bdx \over {|\bdy - \bdx|}}
+ \eta(\bx) \mathcal{D}({\bf u})(\bdx, \bdy) + (\eta(\bdy) - \eta(\bx)) 
\mathcal{D}({\bf u})(\bdx, \bdy) \]
we see that for every ${\bf v} \in C^{\infty}_c(\bbR^d;\mathbb{R}^{d})$, 
\begin{equation}\label{integration-by-p}
\cF^{k}(\eta {\bf u},{\bf v}) -  \cF^{k}({\bf u},\eta {\bf v}) = 
([\mathbb{L}\eta] {\bf u},{\bf v})_{L^2(\bbR^d)} - (I_s({\bf u},\eta),{\bf 
v})_{L^2(\bbR^d)}.
\end{equation}
where  for almost all $\bx \in \bbR^d$ the vector valued function is 
\[
I_s({\bf u},v)(\bdx) = \intdm{\bbR^d}{k(\by-\bx)(\eta({\bf x})-\eta({\bf 
y}))\mathcal{D}({\bf u})(\bdx, \bdy) {\bdx - \bdy \over |{\bf x}-\bdy|}} {\by}, 
\]
which is finite via H\"older's inequality.   In the above and hereafter we 
suppress the dependence of $k$ on $r$. The matrix valued function 
$\mathbb{L}{\eta} (\bdx)$ is given by 
\[
\mathbb{L}{\eta} (\bdx) = {\emph p.v.}\intdm{\mathbb{R}^{d}}{ 
k(\bx-\by)(\eta(\bx) - \eta({\bf y}) )\shapetensorbxy}{\by}.  
\]
Let us justify the $L^{2}$ inner products in the right hand side of 
\eqref{integration-by-p}.    
To that end, we introduce the vector field 
\[
{\bf g}:= (\mathbb{L}\eta){\bf u} - I_s({\bf u},\eta),
\]
and show that ${\bf g} \in L^2(\bbR^d;\mathbb{R}^{d})$. In fact, we also 
show that there exists a constant $C>0$ independent of ${\bf u}$ (but depends 
on $\eta$) such that
\begin{equation}\label{eq-L2RegularityProof-gEstimate}
\Vnorm{{\bf g}}_{L^2(\bbR^d)} \leq C \Vnorm{{\bf u}}_{H^{s}}.
\end{equation}
The rest of the argument is similar to that given in  \cite{Biccari-Warma} adjusted to the system case. 
We include it here fore clarity and completeness. 
We begin by noting that $\mathbb{L}{\eta}$ is uniformly bounded in $\mathbb{R}^{d}$. Indeed, 
using the fact that $\eta\in C_c(\mathbb{R}^{d})$ and $k$ is even, we can 
easily show that $\|\mathbb{L}{\eta}\|_{\infty} \leq C 
(\|D^{2}\eta\|_{L^{\infty}}  + \|\eta\|_{L^{\infty}}.  )$
As a consequence of this and the Poincar\'e-Korn inequality, since ${\bf u} \in 
H_{\Omega}^{s}(\mathbb{R}^{d}, \mathbb{R}^{d})$ we have 
\begin{equation}
\Vnorm{(\mathbb{L}\eta){\bf u}}^2_{L^2(\bbR^d)}  
\leq \Vnorm{\mathbb{L}\eta}^2_{L^{\infty}(\bbR^d)} \Vnorm{{\bf 
u}}^2_{L^2(\bbR^d)} \leq C |{\bf u}|^{2}_{H^{s}}.
\end{equation}
To bound the $L^2$ norm of  $I_s({\bf u},\eta)(\bx) $, we begin by breaking the 
region of integration as 
\begin{align*}
I_s({\bf u},\eta)(\bx) 
&= \intdm{\Omega}{k(\by-\bx)(\eta({\bf x})-\eta({\bf y}))\mathcal{D}({\bf 
u})(\bdx, \bdy) {\bdx - \bdy \over |{\bf x}-\bdy|}} {\by} \\ 
&\quad + \intdm{\bbR^d 
\setminus \Omega}{k(\by-\bx)(\eta({\bf x})-\eta({\bf y}))\mathcal{D}({\bf 
u})(\bdx, \bdy) {\bdx - \bdy \over |{\bf x}-\bdy|}} {\by} \\
&=\intdm{\Omega}{k(\by-\bx)(\eta({\bf x})-\eta({\bf y}))\mathcal{D}({\bf 
u})(\bdx, \bdy) {\bdx - \bdy \over |{\bf x}-\bdy|}} {\by} \\
& \quad + \eta({\bf x}) \intdm{\bbR^d \setminus 
\Omega}{k(\by-\bx)\mathcal{D}({\bf u})(\bdx, \bdy) {\bdx - \bdy \over |{\bf 
x}-\bdy|}} {\by} \\
&:= I_1(\bx) + I_2(\bx).
\end{align*}
Let us estimate the first integral $I_1(\bx)$. Using Cauchy-Schwarz,
\begin{align*}
|I_1(\bx)| 
\leq \left( \intdm{\Omega}{k(\by-\bx)|\eta({\bf x})-\eta({\bf y} )|^{2}} {\by} 
\right)^{1/2} \left( \intdm{\Omega}{k(\by-\bx)|\mathcal{D}({\bf u})(\bdx, 
\bdy)|^{2}} {\by} \right)^{1/2}. 
\end{align*}
Since $\Omega$ is bounded, taking $R= \text{diam}(\Omega),$ we see that for any 
$\bx\in \Omega$, $\Omega\subset B_{2R}(\bdx)$.  
We now use the fact that $\eta$ is smooth and the kernel is comparable with the 
fractional kernel to obtain that 
a constant $C >0$ depending on $\eta$ such that for any $\bdx \in \Omega$
\[
\intdm{\Omega}{k(\by-\bx)|\eta({\bf x})-\eta({\bf y} )|^{2}} {\by} \leq 
\|\nabla \eta\|_{L^{\infty}}\intdm{B_{2R}(\bdx)}{k(\by-\bx) |\bdy - 
\bdx|^2}{\by}= \|\nabla \eta\|_{L^{\infty}}\intdm{B_{2R}({\bfs 0})}{k({\bfs 
\xi}) |{\bfs \xi}|^2}{{\bfs \xi}} \leq C.
\]
For $\bx\in \complement \Omega$, we use the fact $\eta(\bdx) = 0$, and 
$\text{supp}(\eta) \subset \Omega_{2}$, and that $\delta = 
\text{dist}(\Omega_{2}, \partial  \Omega) >0$ to conclude that 
\[
\intdm{\Omega}{k(\by-\bx)|\eta({\bf x})-\eta({\bf y} )|^{2}} {\by}  = 
\intdm{\Omega_{2}}{k(\by-\bx)|\eta({\bf y} )|^{2}} {\by} \leq \| 
\eta\|^{2}_{L^{\infty}} \intdm{\{|{\bfs \xi}| > \delta\}} {k({\bfs \xi})}{{\bfs 
\xi}} \leq C.
\]
Using the two preceding estimates, we see that there exists a positive constant 
$C>0$, that depends on $\eta$ such that 
\begin{equation}\label{eq-L2RegularityProof-I1Estimate}
\intdm{\bbR^d}{|I_1(\bx)|^2}{\bx} \leq C \iintdm{ \bbR^d\, 
\Omega}{k(\by-\bx)|\mathcal{D}({\bf u})(\bdx, \bdy)|^{2}}{\by}{\bx} \leq C 
|{\bf u}|^2_{H^{s}}.
\end{equation} 
To estimate the $L^{2}$ norm of $I_{2}$, again using Cauchy-Schwarz we get that
\[
|I_2(\bx)|^{2} 
\leq |\eta({\bf x})|^{2} \left( \intdm{\bbR^d \setminus \Omega}{k(\by-\bx)} {\by} \right) \left( \, 
\intdm{\bbR^d \setminus\Omega}{k(\by-\bx)|\mathcal{D}({\bf u})(\bdx, 
\bdy)|^{2}} {\by}  \right).
\]
As before, since $\eta(\bdx) = 0$, and $\text{supp}(\eta) \subset \Omega_{2}$, 
and that $\delta = \text{dist}(\Omega_{2}, \partial  \Omega) >0$, we have that 
the function 
\[
\bdx\mapsto |\eta({\bf x})|^{2}\intdm{\bbR^d \setminus \Omega}{k(\by-\bx)} {\by}
\]
is bounded. Thus,
\begin{equation}\label{eq-L2RegularityProof-I2Estimate}
\intdm{\bbR^d}{|I_2(\bx)|^2}{\bx} = \intdm{\Omega_2}{|I_2(\bx)|^2}{\bx} \leq C 
\iintdm{\Omega_2 \, \complement \Omega}{k(\by-\bx)|\mathcal{D}({\bf u})(\bdx, 
\bdy)|^{2}} {\by} {\bx}
 \leq C|{\bf u}|^{2}_{H^{s}}.
\end{equation}
Therefore, the estimate \eqref{eq-L2RegularityProof-gEstimate} of ${\bf g}$ 
follows from \eqref{eq-L2RegularityProof-I1Estimate} and 
\eqref{eq-L2RegularityProof-I2Estimate}. 
We have shown that $\eta {\bf u}$ is a weak solution to the equation
$$
\mathbb{L}(\eta {\bf u}) = {\bf F} \quad \text{ in } \bbR^d,
$$
where ${\bf F} = \eta {\bf f} + (\mathbb{L}\eta){\bf u} - I_s({\bf u},\eta) \in 
L^2(\bbR^d;\mathbb{R}^{d})$. By \autoref{lem:lma-L2RegularityonRd}, $\eta {\bf 
u} \in H^{2s,2}(\bbR^d;\mathbb{R}^{d})$. Thus $\bu \in H^{2s,2}_{loc}(\Omega)$ 
and the proof is complete.
\end{proof}
\subsection{Interior $\mathcal{L}^{2s,p}$  Regularity for $p>2$}
In this section we prove Theorem \autoref{thm:int-for-p}.  As before we start 
with an optimal regularity estimate for the system of equations posed on 
$\mathbb{R}^{d}$. 
\begin{lemma}\label{lem:lma-LpRegularityonRd}
Let $p\in (1, \infty)$. For ${\bf f} \in L^p(\bbR^d;\mathbb{R}^{d})\cap 
L^2(\bbR^d;\mathbb{R}^{d})$, if  ${\bf u} \in H^{s}_{\Omega}(\mathbb{R}^{d}; 
\mathbb{R}^{d})\cap L^p(\bbR^d;\mathbb{R}^{d})$ is a weak solution of
$$
(-\mathring{\mathbf{\Delta}})^{s}{\bf u} = {\bf f} \quad \text{ in } \bbR^d.
$$
Then ${\bf u} \in \cL^{2s,p}(\bbR^d;\mathbb{R}^{d})$, Moreover, there exists a 
constant $C = C(d, s)>0$ such that 
\[
\|(-\mathring{\bfs \Delta})^s{\bf u}\|_{L^{p}} \leq C\|{\bf f}\|_{L^{p}}. 
\]
\end{lemma} 
\begin{proof}
We proceed as in the proof of \autoref{lem:lma-L2RegularityonRd} to obtain 
that in the Fourier space the equation is $
\mathbb{M}({\bfs \xi}) \hat{\bf u} ({\bfs \xi}) = \hat{\bf f}({\bfs \xi})
$ almost everywhere,
where $\mathbb{M}({\bfs \xi})$ is as given in \eqref{formula-for-matrix-symbol} 
with $k$ replaced by the kernel $\frac{1}{|\bx-\by|^{d+2s}}$. For the 
particular form of the kernel, we can explicitly compute the matrix of symbols. 
$\bbM ({\bfs \xi})$ is given by 
\[
\mathbb{M}({\bfs \xi}) = |{\bfs \xi}|^{2s} \left( \ell_1 \bbI + \ell_2 {{\bfs 
\xi}\over |{\bfs \xi}|}\otimes {{\bfs \xi}\over |{\bfs \xi}|} \right)\,,
\]
where $\ell_{1}$ and $\ell_{2}$ are positive numbers given by the formula 
$\ell_i = \int_{\mathbb{R}^{d}} \frac{1-\cos(2\pi h_1)}{|{\bf h}|^{d+2s}} 
{h_{i}^{2} \over |{\bf h}|^{2}} \, \mathrm{d}{\bf h},$ for $i=1, 2$, and 
$\mathbb{I}$ is the $d\times d$ identity matrix. 
The matrix $\ell_1 \bbI + \ell_2 {{\bfs \xi}\over |{\bfs \xi}|}\otimes {{\bfs 
\xi}\over |{\bfs \xi}|} $ is invertible for any ${\bfs \xi} \neq {\bfs 0}$, 
with 
\[
\left( \ell_1 \bbI + \ell_2 {{\bfs \xi}\over |{\bfs \xi}|}\otimes {{\bfs 
\xi}\over |{\bfs \xi}|} \right)^{-1} = \left( {1\over \ell_1} \bbI - 
{\ell_2\over \ell_1(\ell_1 + \ell_2)} {{\bfs \xi}\over |{\bfs \xi}|}\otimes 
{{\bfs \xi}\over |{\bfs \xi}|} \right) 
\]
Using this formula, we can rewrite $\mathbb{M}({\bfs \xi}) \hat{\bf u} ({\bfs 
\xi}) = \hat{\bf f}({\bfs \xi})$ as 
\[
|{\bfs \xi}|^{2s}\hat{\bu}({\bfs \xi}) =  \left( {1\over \ell_1} \bbI - 
{\ell_2\over \ell_1(\ell_1 + \ell_2)} {{\bfs \xi}\over |{\bfs \xi}|}\otimes 
{{\bfs \xi}\over |{\bfs \xi}|} \right) \hat{{\bf f}}({\bfs \xi})
\]
The conclusion of the lemma now follows from the assumption that ${\bf u}\in 
L^{p}(\mathbb{R}^{d}; \mathbb{R}^{d})$ and for any real numbers $a$ and $b$, 
the matrix multiplier
$a \bbI + b  {{\bfs \xi}\over |{\bfs \xi}|}\otimes {{\bfs \xi}\over |{\bfs 
\xi}|}
$
is a $L^p$-multiplier for any $1 < p < \infty$.  The latter follows immediately 
from the $L^p$-boundedness of the Riesz Transforms.

\end{proof}
We use this regularity theorem to prove the second main result of this 
section. Let begin by review the standard fractional Sobolev spaces.
For $p \in [1,\infty)$, $\Omega$ an open subset of $\mathbb{R}^{d}$ and $s \in 
(0,1)$, we define 
$$
W^{s,p}(\Omega;\mathbb{R}^{d}) := \left\{ {\bf u} \in L^p(\Omega) \, : \, 
\iintdm{\Omega \, \Omega}{\frac{|{\bf u}(\bx)-{\bf 
u}(\by)|^p}{|\bx-\by|^{d+sp}}}{\bx}{\by} < \infty \right\}. 
$$
With norm $\|{\bf u}\|^{p} = \|{\bf u}\|_{L^p} + \iintdm{\Omega \, 
\Omega}{\frac{|{\bf u}(\bx)-{\bf u}(\by)|^p}{|\bx-\by|^{d+sp}}}{\bx}{\by}$, it 
is well known that $W^{s,p}(\Omega;\mathbb{R}^{d})$ is a Banach space. 

If $s>1$, then let we write $s = m + \sigma $, where $m$ is the largest integer 
less than $s$, and define 
\[
W^{s,p}(\Omega;\mathbb{R}^{d}) = \{u\in W^{m, p}(\Omega;\mathbb{R}^{d}): 
D^{\alpha}{\bf u} \in W^{\sigma, p}(\Omega;\mathbb{R}^{d}),\quad |\alpha| = 
m\}. 
\]
With the norm $\|{\bf u}\|^{p} = \|{\bf u}\|_{W^{m,p}} + \sum_{|\alpha| = m} 
\|D^{\alpha}\|_{W^{\sigma, p}}$, the space $W^{s,p}(\Omega;\mathbb{R}^{d})$
is known to be a Banach space.  

We also need the Sobolev embedding result 
\[
W^{r, p}(\mathbb{R}^{d};\mathbb{R}^{d}) \hookrightarrow W^{s, 
q}(\mathbb{R}^{d};\mathbb{R}^{d}),\,\text{provided $0 < s\leq r$, $1 <p\leq 
q<\infty$, and $r- {d\over p} = s-{d\over q}$}
\]
If $\Omega$ is open and bounded with smooth enough boundary then 
\[
W^{r, p}(\Omega;\mathbb{R}^{d}) \hookrightarrow W^{s, 
q}(\Omega;\mathbb{R}^{d}),\,\text{provided $0 < s\leq r$, $1 <p\leq q<\infty$, 
and $r- {d\over p} \geq s-{d\over q}$}. 
\]
Let us recall the relation between the potential spaces and the fractional 
Sobolev spaces, \cite[Chapter~5, Theorem~5]{Stein}. For our purpose it suffices 
to recall 
\[
\mathcal{L}^{s, p}(\mathbb{R}^{d};\mathbb{R}^{d}) \subset 
W^{s,p}(\mathbb{R}^{d};\mathbb{R}^{d}), \,\text {if $p \geq 2$.}
\] 
For $p=2$, the spaces are the same; $H^{s}(\mathbb{R}^{d};\mathbb{R}^{d}) = 
\mathcal{L}^{s, 2}(\mathbb{R}^{d};\mathbb{R}^{d}) = W^{s, 
2}(\mathbb{R}^{d};\mathbb{R}^{d})$.

\begin{proof}[Proof of \autoref{thm:int-for-p}]
Let ${\bf f} \in L^p_{\Omega}(\bbR^d)$. Since $p \geq 2$, and $\Omega$ is 
bounded, ${\bf f} \in L^2_{\Omega}(\bbR^d)$. Therefore a unique weak solution 
${\bf u}$ in $H^{s}_{\Omega}(\mathbb{R}^{d})$ exists.  
Let $\eta$ be the cutoff function constructed in the proof 
of \autoref{thm:InteriorRegularity-L2MainTheorem}. Then we have that  $\eta 
{\bf u} \in H^{2s}(\mathbb{R}^{d};\mathbb{R}^{d}))= W^{2s,2}(\bbR^d)$. By 
Sobolev Embedding, $\eta {\bf u} \in W^{s,2^{*_{s}}}(\bbR^d;\mathbb{R}^{d})$.
Now, let $\omega_1$, $\omega_2$ be open sets such that $\omega \Subset \omega_1 
\Subset \omega_2 \Subset \Omega$. The cutoff function $\eta$ was arbitrary, so 
therefore ${\bf u} \in W^{s,2^{*_{s}}}(\omega_2;\mathbb{R}^{d})$. 

{\bf Part a.} ($p \leq 2^{*_{s}}$) Since $\omega_2$ is bounded and $p \leq 2^{*_{s}}$, again by Sobolev Embedding  
${\bf u} \in W^{s,p}(\omega_2;\mathbb{R}^{d})$. 
Moreover, since ${\bf u}\in H^{s}_{\Omega}(\mathbb{R}^{d};\mathbb{R}^{d})$, we 
have that ${\bf u}\in W^{s;2}(\Omega) \hookrightarrow L^p_{\Omega}(\bbR^d)$.  
In summary, ${\bf u} \in H^{s}(\mathbb{R}^{d};\mathbb{R}^{d}) \cap 
W^{s,p}(\omega_2;\mathbb{R}^{d}) \cap L^p_{\Omega}(\bbR^d;\mathbb{R}^{d})$. 
Now we proceed as in the proof of 
\autoref{thm:InteriorRegularity-L2MainTheorem}. With the same reasoning, 
$\eta{\bf u} $ is a weak solution of 
\[
(-\mathring{{\bfs \Delta}})^{s}(\eta {\bf u}) = \mathbf{F} \quad \text{in 
$\mathbb{R}^{d}$}, 
\]
where $\mathbf{F} = \eta {\bf f} + ((-\mathring{{\bfs \Delta}})^{s}\eta){\bf u} 
- I_s({\bf u},\eta).$  
Let ${\bf g} := ((-\mathring{{\bfs \Delta}})^{s}\eta){\bf u} - I_s({\bf 
u},\eta)$. We have shown already that ${\bf g}\in 
L^{2}(\mathbb{R}^{d};\mathbb{R}^{d})$.  Noting that $\eta{\bf u}\in 
H^{s}(\mathbb{R}^{d};\mathbb{R}^{d})\cap L^p_{\Omega}(\bbR^d;\mathbb{R}^{d})$, 
we can now apply \autoref{lem:lma-LpRegularityonRd} to conclude that 
$\eta{\bf u}\in \mathcal{L}^{2s, p}(\mathbb{R}^{d};\mathbb{R}^{d})$ provided we 
successfully show 
${\bf g} \in L^p(\bbR^d;\mathbb{R}^{d})$. In fact, we demonstrate that for 
some constant $C>0$ independent of ${\bf u}$ 
\begin{equation}\label{eq-LpRegularityProof-gEstimate}
\Vnorm{{\bf g}}_{L^p(\bbR^d)} \leq C (\Vnorm{{\bf 
u}}_{W^{s,p}(\omega_2;\mathbb{R}^{d})} + \Vnorm{{\bf 
u}}_{L^p(\Omega;\mathbb{R}^{d})}).
\end{equation}
The last estimate follows from a similar argument as in 
\autoref{thm:InteriorRegularity-L2MainTheorem}. We sketch its proof. More 
detail can be found in \cite{Biccari-Warma}.  
As before, the matrix function $(-\mathring{{\bfs \Delta}})^{s}\eta \in 
L^{\infty}(\bbR^d)$. Thus,
\begin{align*}
\Vnorm{((-\mathring{{\bfs \Delta}})^{s}\eta){\bf u}}^p_{L^p(\bbR^d)} &= 
\intdm{\bbR^d}{\left| \intdm{\bbR^d}{\frac{\eta(\bx)-\eta(y)}{|\bx-\by|^{d+2s}} 
\shapetensorbxy}{\by} \, {\bf u}(\bx) \right|^p}{\bx} \\
&\leq \Vnorm{(-\mathring{{\bfs \Delta}})^{s}\eta}^p_{L^{\infty}(\bbR^d)} 
\Vnorm{{\bf u}}^p_{L^p(\Omega)}.
\end{align*}
The second term $I_s({\bf u},\eta)$ can also be dealt with in the same way as 
in the proof of \autoref{thm:InteriorRegularity-L2MainTheorem}. We begin by 
breaking the integral as 
\begin{align*}
I_s({\bf u},\eta)(\bx) &= 
\intdm{\omega_1}{\frac{\eta(\bx)-\eta(\by)}{|\bx-\by|^{d+2s}} \mathcal{D}{\bf 
u}(\bdx, \bdy)}{\by} +  \eta(\bx) \intdm{\bbR^d \setminus 
\omega_1}{\frac{\mathcal{D}({\bf u})(\bdx, \bdy)}{|\bx-\by|^{d+2s}}  }{\by} := 
I_1(\bx) + I_2(\bx).
\end{align*}
We estimate $I_1(\bx)$: Using Holder's inequality with conjugate $p' = p/(p-1)$,
\begin{align*}
|I_1(\bx)| &\leq \intdm{\omega_1}{\frac{|\eta(\bx)-\eta(\by)|}{|\bx-\by|^{d+2s}} 
\diffqbunorm}{\by} \\
&\leq \left( 
\intdm{\omega_1}{\frac{|\eta(\bx)-\eta(\by)|^{p'}}{|\bx-\by|^{d+sp'}}}{\by} 
\right)^{1/p'} \left( \intdm{\omega_1}{\diffqbunorm^p 
\frac{1}{|\bx-\by|^{d+sp}}}{\by} \right)^{1/p},
\end{align*}
from which we get that
\begin{equation}\label{eq-LpRegularityProof-I1Estimate}
\begin{split}
\intdm{\bbR^d}{|I_1(\bx)|^p}{\bx} 
\leq C \left( \Vnorm{{\bf u}}^p_{W^{s,p}(\omega_2)} + \Vnorm{{\bf u}}^p_{L^p(\Omega)} \right).
\end{split}
\end{equation}
Similarly, we also get that 
\begin{equation}\label{eq-LpRegularityProof-I2Estimate}
\begin{split}
\intdm{\bbR^d}{|I_2(\bdx)|^p}{\bdx} 
\leq C \Vnorm{{\bf u}}^p_{L^p(\Omega)}.
\end{split}
\end{equation}
Therefore, the estimate \eqref{eq-LpRegularityProof-gEstimate} of ${\bf g}$ follows from \eqref{eq-LpRegularityProof-I1Estimate} and \eqref{eq-LpRegularityProof-I2Estimate}.  

{\bf Part b.} ($p > 2^{*_{s}}$)  From Part a) we have that ${\bf u}\in W^{2s, 2^{*_{s}}}_{loc}(\Omega;\mathbb{R}^{d})$, and so ${\bf u}\in W^{2s, 2^{*_{s}}}(\omega_2;\mathbb{R}^{d})$. By Sobolev Embedding we have that ${\bf u}\in W^{s, q_1}(\omega_2;\mathbb{R}^{d})$ with $q_1 = \min\left\{p, \frac{N 2^{*_{s}}}{N-s 2^{*_{s}}}\right\} = \min\left\{p, \frac{2N}{N-4s}\right\}$. By assumption ${\bf u} \in L^{q_1}(\Omega;\mathbb{R}^{d})$. With this information, we can now repeat the argument in Part a) to conclude that ${\bf u}\in W^{2s, q_1}_{loc}(\Omega;\mathbb{R}^{d})$. Now, if $2\leq p\leq \frac{2N}{N-4s},$ the proof is completed. Otherwise we iterate the above procedure to obtain ${\bf u}\in W^{2s, q_j}_{loc}(\Omega;\mathbb{R}^{d})$ with $q_j = \min\left\{p, \frac{2N}{N-js}\right\}$, for all $j\geq 2.$ For $p\geq 2^{*_{s}}$ , we can now choose $j \in \mathbb{N}$ such that $2\leq p \leq  \frac{2N}{N-js}$.  That completes the proof of the theorem. 
\end{proof}


\end{document}